\documentclass[]{amsart}
\usepackage[utf8]{inputenc}
\usepackage[T1]{fontenc}
\usepackage{amsmath,amssymb, amsthm, graphicx, tikz, mathtools, mathrsfs, comment, setspace, microtype,dsfont,float,hyperref,tikz-cd,enumitem,accents,xspace,bm,subcaption}
\usepackage[justification=centering]{caption}

\usepackage{biblatex}
\newtheorem{theorem}{Theorem}[section]
\newtheorem{corollary}[theorem]{Corollary}
\newtheorem{lemma}[theorem]{Lemma}
\newtheorem{proposition}[theorem]{Proposition}
\newtheorem{definition}[theorem]{Definition}

\newcommand{\R}{\mathbb{R}}
\newcommand{\C}{\mathbb{C}}

\newcommand{\Abs}[1]{\left\lvert #1 \right\rvert}
\newcommand{\isomorphic}{\cong}

\newcommand{\brackets}[1]{\langle #1 \rangle}
\renewcommand{\bar}[1]{\overline{#1}}

\newcommand{\interior}[1]{\accentset{\circ}{#1}}

\renewcommand{\Re}{\text{\textnormal{Re}}}
\renewcommand{\Im}{\text{\textnormal{Im}}}
\newcommand{\norm}[1]{\left\lVert#1\right\rVert}

\newcommand{\tr}{\text{\textnormal{tr}}}

\newcommand\restr[2]{{
		\left.\kern-\nulldelimiterspace 
		#1 
		\vphantom{\big|} 
		\right|_{#2} 
}}

\begin{document}

\title{The Brown measure of non-Hermitian sums of projections}
\author{Max Sun Zhou}
\date{}

\begin{abstract}
	We compute the Brown measure of the non-normal operators $X = p + i q$, where $p$ and $q$ are Hermitian, freely independent, and have spectra consisting of $2$ atoms. The computation relies on the model of the non-trivial part of the von Neumann algebra generated by 2 projections as $2 \times 2$ random matrices. We observe that these measures are supported on hyperbolas and note some other properties related to their atoms and symmetries. 
\end{abstract}

\maketitle

\section{Introduction}
Let $M$ be a von Neumann algebra with a faithful, normal, tracial state $\tau$. Let $X \in (M, \tau)$. The Brown measure of $X$, introduced in \cite{BrownPaper}, is a complex Borel probability measure supported on the spectrum of $X$. When $X$ is normal, the Brown measure of $X$ is the spectral measure. When $X$ is a random matrix, the Brown measure of $X$ is the empirical spectral distribution of $X$.

The first interesting explicit computations for the Brown measure of non-trivial operators were provided in \cite{HaagerupPaper1}, nearly 20 years after the Brown measure was first introduced. This class of operators are the ``$R$-diagonal operators'' which include Voiculescu's circular operator. In recent years, there has been much research in computing the Brown measure for a wide variety of families of operators (see \cite{HoPaper}, \cite{HoZhongPaper}, and \cite{BelinschiPaper1}  for some examples). 

In this paper, we will explicitly compute the Brown measure of operators of the form $X = p + iq$, where $p, q \in (M, \tau)$ are Hermitian, freely independent, and have spectra consisting of $2$ atoms. Note that when $p$ or $q$ only has $1$ atom in their spectra (i.e. is a real number), then $X$ is normal, so the Brown measure can be trivially computed. In the course of studying these operators, we will deduce that when $p$ and $q$ have $2$ atoms, then $X$ is non-normal. 

The key tool is the model of the von Neumann algebra generated by $2$ projections from \cite{VoiculescuPaper} to study the non-trivial part of this space. This reduces the computation in $(M, \tau)$ to $\left( M_2 \left( L^{\infty} ((0, 1), \nu^*) \right) , \mathbb{E}_{\nu^*}[ \frac{1}{2} \tr ]  \right)$ for some Borel probability measure $\nu^*$ on $(0, 1)$. To completely determine the Brown measure, we use freeness of $p$ and $q$ to explicitly compute $\nu^*$ and some other relevant parameters.

We will show these measures are supported on hyperbolas in the complex plane and observe some other interesting properties about their atoms and symmetries. 

One important aspect of the Brown measure is that it is the candidate for the limit of the empirical spectral distribution for random matrices: Given a random matrix model $X_n \in M_n(\C)$, suppose that $X_n$ converges in $*$-moments to $X \in (M, \tau)$. The Brown measure of $X$ is typically the limit of the empirical spectral distributions of the $X_n$ (see \cite{SniadyPaper} for a precise statement of this). However, this does not always hold, as \cite[][Chapter 11, Exercise 15]{SpeicherBook} provides a simple counterexample. On the other hand, two notable results that demonstrate this principle are the Circular Law in \cite{TaoPaper} and the Single Ring Theorem in \cite{GuionnetPaper}. 

In a following paper, we will consider a natural random matrix model $X_n$ corresponding to $X = p + i q$ and prove that the empirical spectral distributions of the $X_n$ converge to the Brown measure of $X$. 

The rest of the paper is organized as follows. In Section \ref{sec:preliminaries}, we recall the definition of the Brown measure, the model of the algebra generated by $2$ projections, and some free probability transforms. In Section \ref{sec:notation}, we fix some notation and give an outline of the steps of the computation. In Section \ref{sec:brown_measure_compute}, we compute the Brown measure of $X$ up to some parameters that depend on the joint law of $p$ and $q$ (Proposition \ref{prop:brown_measure_p+iq_no_nu}). In Section \ref{sec:atoms_weights}, we use the freeness of $p$ and $q$ to compute these parameters (Propositions \ref{prop:atoms_weights} and \ref{prop:nu_measure}) and also deduce that $X$ is non-normal (Corollary \ref{cor:projections_not_normal}). In Section \ref{sec:full_brown_measure}, we combine these results to state the Brown measure of $X$ (Theorem \ref{thm:brown_measure_p+iq}). We also deduce some properties and provide accompanying figures to illustrate these properties. 


\section{Preliminaries}
\label{sec:preliminaries}

\subsection{Brown measure}
In this subsection, we recall the definition of the Brown measure of an operator $X \in (M, \tau)$ and describe how to compute it. 

First, we recall the definition of the Fuglede-Kadison determinant, first introduced in \cite{KadisonPaper}: 

\begin{definition}
	Let $x \in (M, \tau)$. Let $\mu_{\Abs{x}}$ be the spectral measure of $\Abs{x} = (x^* x)^{1/2}$. Then, the \textbf{Fuglede-Kadison determinant} of $x$, $\Delta(x)$, is given by: 
	\begin{equation}
		\Delta(x) = \exp \left[ \int_{0}^{\infty} \log t \, d \mu_{\Abs{x}}(t) \right]   \,.
	\end{equation}
\end{definition}

When $x = X_n \in (M_n(\C), \frac{1}{n} \tr)$, then $\Delta(X_n) = \Abs{\det(X_n)}^{1/n}$. Thus, the Fuglede-Kadison determinant is a generalization of the normalized, positive determinant of a complex matrix. 

By applying the functional calculus to a decreasing sequence of continuous functions on $[0, \infty)$ converging to $\log t$, we have: 

\begin{equation}
	\log \Delta(x) = \int_{0}^{\infty} \log t \, d \mu_{\Abs{x}}(t) = \frac{1}{2} \int_{0}^{\infty} \log t \, d \mu_{\Abs{x}^2}(t) \,.
\end{equation}

In practice, it is more convenient to compute the right-hand side to compute $\Delta(X)$. 

Then, the Brown measure is the following distribution:

\begin{definition}
	Let $x \in (M, \tau)$. The \textbf{Brown measure} of $x$ is defined as:
	\begin{equation}
		\mu_x = \frac{1}{2 \pi}\nabla^2 \log \Delta(z - x) \,.
	\end{equation}
\end{definition} 

In \cite{BrownPaper}, some properties of the Brown measure are proven (see also \cite{SpeicherBook}, \cite{TaoBook}, \cite{HaagerupPaper} for more exposition on some of the basic results). In particular, the Brown measure of $x$ is a probability measure supported on the spectrum of $x$. When $x$ is normal, then the Brown measure is the spectral measure. When $x$ is a random matrix, then the Brown measure is the empirical spectral distribution.

\subsection{The von Neumann algebra generated by $2$ projections}
\label{subsec:two_projections}

In this subsection, we recall some results from \cite[][Section 12]{VoiculescuPaper} characterizing the tracial von Neumann algebra generated by $2$ projections. Let $(M, \tau)$ be the von Neumann algebra generated by two projections $p$ and $q$. 

Let $(M, \tau)$ be the von Neumann algebra generated by two projections $p$ and $q$. 

Consider the following elements of $M$: 
\begin{equation}
	\label{eqn:projections_e}
	\begin{aligned}
		e_{0 0} & = (1 - p) \wedge (1 - q) \\
		e_{0 1} &= (1 - p) \wedge q \\
		e_{1 0} &= p \wedge (1 - q) \\
		e_{1 1} &= p \wedge q \\
		e &= 1 - (e_{0 0} + e_{0 1} + e_{1 0} + e_{1 1}  ) \,.
	\end{aligned}
\end{equation}
These elements are central, mutually orthogonal projections and 
\begin{equation}
	e_{0 0} + e_{0 1} + e_{1 0} + e_{1 1} + e = 1 \,.
\end{equation}
If we consider the matrix of $m \in M$ with respect to the $e_{i j}, e$ (i.e. the $5 \times 5$ matrix with entries of the form $f m f'$ where $f, f' \in \{ e_{i, j}, e  \}$), then this matrix is diagonal. Further, the 4 diagonal terms $e_{i j} m e_{i j}$ are constants. Hence, the only interesting part of $M$ is the subalgebra $e M e$. In particular,
\begin{equation}
	\begin{aligned}
		p & = p \wedge (1 - q) + p \wedge q  + e p e = e_{1 0} + e_{1 1} + e p e \\
		q &= (1 - p) \wedge q + p \wedge q + e q e = e_{0 1} + e_{1 1} + e q e \,.
	\end{aligned}
\end{equation}
When $e \neq 0$, consider $(e M e, \restr{\tau}{eM e})$ as a von Neumann subalgebra with identity $e$ and $\restr{\tau}{eM e}(eme) = \frac{\tau(eme)}{\tau(e)} $. Then, 
\begin{equation}
	(e M e, \restr{\tau}{eM e}) \isomorphic \left( M_2 \left( L^{\infty} ((0, 1), \nu^*) \right) , \mathbb{E}_{\nu^*}[ \frac{1}{2} \tr ]  \right)          \,.
\end{equation}
where $\nu^*$ is a Borel measure on $(0, 1)$. 

For any $m \in M$, let $\tilde{m} = e m e \in e M e$. The isomorphism has the following correspondence between $\tilde{m} \in e M e$ and matrix-valued functions of $t \in (0, 1)$: 
\begin{equation}
	\begin{aligned}
		\tilde{p} & \leftrightarrow 
		\begin{pmatrix}
			t & (t - t^2)^{1/2} \\
			(t - t^2)^{1/2} & 1 - t
		\end{pmatrix} \\ 
		\tilde{q} & \leftrightarrow 
		\begin{pmatrix}
			1 & 0 \\
			0 & 0
		\end{pmatrix} .
	\end{aligned} 
\end{equation}
The measure $\nu^*$ is recovered by noting that for the central element $x = p q p + (1 - p)(1 - q)(1 - p) = 1 - p - q + pq + qp$, 
\begin{equation}
	\tilde{x} \leftrightarrow 
	\begin{pmatrix}
		t & 0 \\
		0 & t
	\end{pmatrix} .
\end{equation}
Hence, $\nu^*$ is the spectral measure of the element $\tilde{x}$ in $(e M e, \restr{\tau}{eM e})$.

Observe that with the change of variables $t = \cos^2 \theta$, $\theta \in (0, \pi / 2)$, 
\begin{equation}
	\tilde{p} = R_\theta \tilde{q} R_\theta^{-1} \,,
\end{equation}
where $R_\theta \in M_2(\C)$ is the rotation matrix with angle $\theta$. 

In practice, we will work with $\theta$ instead of $t$. Let $(\cos^2)^{-1}: [0, 1] \to [0, \pi / 2]$ be the inverse of $\cos^2 : [0, \pi / 2] \to [0, 1]$ and consider the measure $\nu$ on $(0, \pi / 2)$ given by:
\begin{equation}
	\nu = \left( (\cos^2)^{-1}\right) _{*}(\nu^*) \,.
\end{equation}
Then, we will use the isomorphism:
\begin{equation}
	(e M e, \restr{\tau}{eM e}) \isomorphic \left( M_2 \left( L^{\infty} ((0, \pi / 2), \nu) \right) , \mathbb{E}_{\nu}[ \frac{1}{2} \tr ]       \right) \,.
\end{equation}
Let us highlight a few facts about this algebra. First, we explain why in the matrix algebras the domains are open instead of closed. This is because of the following fact: 

\begin{lemma}
	\label{lem:nu_0_1}
	The measure $\nu^*$ does not have atoms at $0$ or $1$, i.e. $\nu^* (\{0, 1\}) = 0$. This is equivalent to $\nu$ not having atoms at $0$ or $\pi / 2$.
\end{lemma}
\begin{proof}
	Let $\tilde{p}  = e p e$ and $\tilde{q} = e q e$. Since $e$ commutes with $p$ and $q$, $\tilde{p} = p \wedge e$ and $\tilde{q} = q \wedge e$. 
	
	Consider $\tau( (\tilde{q} \tilde{p} \tilde{q})^n) \to \tau(\tilde{p} \wedge \tilde{q}) = \tau(e \wedge (p \wedge q))$ as $n \to \infty$. Recall that $e$ and $p \wedge q$ are mutually orthogonal, so $\tau(e \wedge (p \wedge q)) = 0$. Under the isomorphism, computation shows that
	\begin{equation}
		\tau((\tilde{q} \tilde{p} \tilde{q})^n) = \int_{0}^{1} t^n \, d \nu^*(t) \,.
	\end{equation}
	As $n \to \infty$ the integral on the right-hand side decreases to $\nu^*(\{1\})$. Hence, $\nu^*(\{1\}) = 0$.
	
	A similar argument considering $\tau(((1 - \tilde{q})(\tilde{p})(1 - \tilde{q}))^n) \to \tau((1 -\tilde{ q}) \wedge \tilde{p} )$ using that $e$ and $(1 - q) \wedge p$ are mutually orthogonal shows that $\nu^* (\{0\}) = 0$.	
	
	Finally, we note that from the definition of pushforward measure, $\nu^*(\{0, 1\}) = \nu(\{0, \pi / 2\})$.
\end{proof}

Next, we note that on $e M e$, $e p e$ and $e q e$ have trace $1/2$. This follows from $e$ being mutually orthogonal to all of the $e_{i j}$:

\begin{lemma}
	\label{lem:trace_1/2}
	Let $\tau$ be the trace on $e M e $. Then, 
	$\tau(e p e) = \tau(e q e) = 1 / 2$.
\end{lemma}
\begin{proof}
	If this is false, then we may choose one of $epe$ or $1 - epe$ and one of $eqe$ or $1 - eqe$ so that the sum of the traces is greater than $1$. Without loss of generality assume that $\tau(e p e) + \tau(e q e) > 1$. Since $e$ commutes with $p$ and $q$, $e p e = p \wedge e$ and $e q e = q \wedge e$.  Then, from the parallelogram law, we have the following contradiction:
	\begin{equation}
		\begin{aligned}
			0 & = \tau( e \wedge (p \wedge q)) = \tau(e p e \wedge e q e) = \tau(epe) + \tau(eqe) - \tau(e p e \vee eqe)   
			\\ & \geq \tau(epe) - \tau(eqe) - 1 > 0 \,.
		\end{aligned}
	\end{equation}
\end{proof}

\subsection{Free probability transforms} 
\label{subsec:free_prob_functions}
In this final preliminary subsection, we review some definitions and properties of free probability transforms. First, we recall that the definition of freeness in $(M, \tau)$: 

\begin{definition}
	Let $(M, \tau)$ be a tracial von Neumann algebra, and let $\{A_i\}_{i \in I}$ be a family of unital subalgebras of $M$. $\{A_i\}_{i \in I}$ are \textbf{freely independent } if for any $a_j \in A_{k(j)}$ with $k(j) \neq k(j + 1)$, $j = 1, \ldots, n = 1$ and $\tau(a_i) = 0$, then 
	\begin{equation}
		\tau(a_1 \ldots, a_n) = 0 \,.
	\end{equation}
	Let $r, (m_k)_{1 \leq k \leq r}$ be positive integers. The sets $\{X_{1, p}, \ldots, X_{m_p, p}  \}_{1 \leq p \leq r}$ of non-commutative random variables are \textbf{free} if the algebras they generate are free.
\end{definition}

If $A_i$ generate $(M, \tau)$ and are freely independent, then $\restr{\tau}{A_i}$ determines $\tau$.

Now, we recall the definition of some free probability transforms of real measures and state some of their basic properties. We refer to \cite{GuionnetBook}, \cite{SpeicherBook}, and \cite{TaoBook} for proofs of these facts.

For a Hermitian $x \in (M, \tau)$, we will abuse notation by using the subscript $x$ to denote the integral transform with respect to the spectral measure $\mu_x$.

First, we define the Stieltjes transform of a real probability measure: 

\begin{definition}
	Let $\mu$ be a probability measure on $\R$. Then, the \textbf{Stieltjes transform} of $\mu$ is the function $G_\mu: \C \setminus \text{supp}(\mu) \to \C$ given by: 
	\begin{equation} 
		G_\mu(z) = \int_{\R}^{} \frac{1}{z - t} \, d \mu(t) \,.
	\end{equation}
\end{definition}

We list some of the well-known facts about $G_\mu$ in the following Proposition: 

\begin{proposition}
	Let $G_\mu$ be the Stieltjes transform of $\mu$. Then, 
	\begin{itemize}
		\item $G_\mu$ is analytic on $\C \setminus \text{supp}(\mu)$. 
		\item Let $\mathbb{H}^{\pm }(\C)$ be the upper/lower half-planes of $\C$. Then, $G_\mu: \mathbb{H}^{\pm} \to \mathbb{H}^{\mp }$. In particular, $G_\mu(z) \in \R$ if and only if $z \in \R \setminus \text{supp}(\mu)$.
		\item $\bar{G_\mu(z)} = G_\mu(\bar{z})$ for $z \in \C \setminus \text{supp}(\mu)$.
		\item $\Abs{G_\mu(z)} \leq \frac{1}{\Abs{\Im(z)}}$.
		\item $G_\mu$ is analytic on $\C \setminus \text{supp}(\mu)$. 
		\item $\mu$ is compactly supported, $G_\mu(z)$ has the following Laurent series expansion for $\Abs{z} > \sup_{\lambda \in \text{supp}(\mu) } \Abs{\lambda}$:
		\begin{equation}
			G_\mu(z) = \sum_{n = 0}^{\infty} \left( \int_{\R}^{} t^n \, d \mu(t) \right) z^{- n - 1} \,.
		\end{equation}
	\end{itemize}
\end{proposition}

Thus, for compact measures, the Stieltjes transform of a measure $\mu$ contains the same information as the moments of $\mu$. 

We highlight one special fact about the Stieltjes transform. For $z = a + i b$, note that 
\begin{equation}
	\Im \, \frac{1}{z - t} = - \frac{b}{(t - a)^2 + b^2} \,.
\end{equation}
Recall that the Poisson kernel on the upper half-plane is given by: 
\begin{equation}
	P_b(a) = \frac{1}{\pi } \frac{b}{a^2 + b^2} \, , \quad b > 0 \,.
\end{equation}
Combining these two facts shows that 
\begin{equation}
	- \frac{1}{\pi} \Im \, G_\mu(a + i b)  = (\mu * P_b)(a) \,.
\end{equation}
Recall that $P_b$ are approximations to the identity as $b \to 0^+$, so then 
\begin{equation}
	\label{eqn:stieltjes_vague_limit}
	\lim\limits_{b \to 0^+ }- \frac{1}{\pi} \Im \, G_\mu(\cdot + i b) = \mu 
\end{equation}
in the vague topology on $\R$.

By exploiting the conjugate symmetry of $G_\mu$, we can also write this in terms of the discontinuity of $G_\mu$ across $\R$: 
\begin{equation}
	\lim\limits_{b \to 0^+} - \frac{G_\mu(\cdot + i b) - G_\mu(\cdot - i b)}{2 \pi i} = \mu 
\end{equation}
in the vague topology on $\R$.

Additionally, there is an explicit formula for intervals (\cite{SpeicherBook}, Theorem 6): 

\begin{proposition}
	\label{prop:stieltjes_intervals}
	For $a, b \in \R$ and  $a < b$, 
	\begin{equation}
		\lim\limits_{y \to 0^+} \int_{a}^{b} - \frac{1}{\pi} \Im G_\mu(x + i y) \, d x = \mu((a, b)) + \frac{1}{2} \mu(\{a\}) + \frac{1}{2} \mu(\{b\}) \,.
	\end{equation}
\end{proposition}

When $\mu$ is a compactly supported real measure, then 
\begin{equation}
	\lim\limits_{\Abs{z} \to \infty} z G_\mu(z) = 1 \,.
\end{equation}
Computation using $F_\mu$ shows that $R_\mu$ is analytic in a neighborhood of $0$, as the $1/w$ is exactly the pole of $G_\mu^{\brackets{-1}}$ at $w = 0$.

For non-compactly supported measures $\mu$, then $G_\mu^{\brackets{-1}}$ can still be defined on wedge-shaped domains containing $0$ \cite[see][Theorem 33]{SpeicherBook}, but we will not need this fact.

Similar to the Stieltjes transform, we can define $\psi_\mu$:

\begin{definition}
	Let $\mu$ be a probability measure on $[0, \infty)$. Let $\text{supp}(\mu)^{-1} = \{ 1 / x : x \in \text{supp}(\mu) \}$. Define $\psi_\mu: \C \setminus \text{supp}(\mu)^{-1} \to \C$ by: 
	\begin{equation}
		\psi_\mu(z) = \int_{\R}^{} \frac{t z}{1 - tz} \, d \mu(t) \,.
	\end{equation}
\end{definition}

We list some well-known properties of $\psi_\mu$ in the following Proposition: 

\begin{proposition}
	Let $\psi_\mu$ be as above. Then, 
	\begin{itemize}
		\item $\psi_\mu(0) = 0$.
		\item $\psi_\mu$ is analytic on $\C \setminus \text{supp}(\mu)^{-1}$.
		\item $\bar{\psi_\mu(z)} = \psi_\mu(\bar{z})$ for all $z \in \C \setminus \text{supp}(\mu)^{-1}$.
		\item  Let $\mathbb{H}^{\pm }(\C)$ be the upper/lower half-planes of $\C$. If $\mu(\{0\}) < 1$, then $\psi_\mu: \mathbb{H}^{\pm }(\C) \to \mathbb{H}^{\pm }(\C)$. In particular, $\psi_\mu(z) \in \R$ if and only if $z \in \R \setminus \text{supp}(\mu)^{-1}$.
		\item If $\mu$ is compactly supported, then $G_\mu(z)$ has the following Taylor series expansion for $\Abs{z} < \inf_{\lambda \in \text{supp}(\mu)^{-1}} \Abs{\lambda}$:
		\begin{equation}
			\psi_{\mu}(z) = \sum_{n = 1}^{\infty} \left(  \int_{\R}^{} t^n \, d \mu(t) \right) z^n \,.
		\end{equation}
	\end{itemize}
\end{proposition}

We highlight the following equation, valid for $z \in \C \setminus \text{supp}(\mu)$: 
\begin{equation}
	G_\mu(z) = \frac{1}{z} \left( \psi_\mu \left( \frac{1}{z} \right) + 1 \right)  \,.
\end{equation}
To consider the inverse of $\psi_\mu$ at $z = 0$, note that 
\begin{equation}
	\psi'(0) = \int_{\R}^{} t \, d \mu(t) \,.
\end{equation}
As long as $\mu(\{0\}) < 1$, then $\psi_\mu$ is invertible in a neighborhood of $0$. In this situation, define the following functions: 

\begin{definition}
	Let $\mu$ be a compactly supported probability measure on $[0, \infty)$ such that $\mu (\{0\}) < 1$. Then, $\psi_\mu$ is invertible in a neighborhood of $0$ in $\C$. Define
	\begin{equation}
		\begin{aligned}
			\chi_\mu(w) & = \psi_\mu^{\brackets{-1}}(w) \\
			S_\mu(w) &= \chi_\mu(w) \frac{w + 1}{w} \,.
		\end{aligned}
	\end{equation}
	$S_\mu$ is called the $\bm{S}$\textbf{-transform} of $\mu$.
\end{definition}

In particular, since $\chi_\mu(0) = 0$ then $S_\mu$ is well-defined in a neighborhood of $0$. 

Finally, we recall the addition (resp. multiplication) laws for free additive (resp. free multiplicative) convolution:

\begin{theorem}
	Let $x, y \in (M, \tau)$ be Hermitian and freely independent. Let $x \boxplus y$ be the spectral measure of $x + y$. Then, where the functions are defined,
	\begin{equation}
		R_{x \boxplus y}(z) = R_{\mu_x}(z) + R_{\mu_y}(z) \,.
	\end{equation}
	If $x, y$ are positive, let $x \boxtimes y$ be the spectral measure of $x^{1/2} y x^{1/2}$. Then, where the functions are defined,
	\begin{equation}
		S_{x \boxtimes y}(z) = S_{\mu_x}(z) S_{\mu_y}(z) \,.
	\end{equation}
\end{theorem}

\section{Notation and Outline of Computation}
\label{sec:notation}
Recall that we consider operators of the form $X = p + i q$, where $p$ and $q$ are Hermitian, freely independent, and have $2$ atoms in their spectra. 

We will use the following notation for the atoms of $p$ and $q$ and their weights: 
\begin{equation}
	\begin{aligned}
		\mu_p & = a \delta_\alpha + (1 - a) \delta_{\alpha'} \\
		\mu_q &= b \delta_\beta + (1 - b) \delta_{\beta'} \,,
	\end{aligned}
\end{equation}
where $a, b \in (0, 1)$, $\alpha, \alpha', \beta, \beta' \in \R$, $\alpha \neq \alpha'$, and $\beta \neq \beta'$.

We fix $p'$ (resp. $q'$) to be the following spectral projection of $p$ (resp. $q$):

\begin{definition}
	Let $p, q \in (M, \tau)$ be Hermitian with laws:
	\begin{equation}
		\begin{aligned}
			\mu_p & = a \delta_\alpha + (1 - a) \delta_{\alpha'} \\
			\mu_q &= b \delta_\beta + (1 - b) \delta_{\beta'} \,.
		\end{aligned}
	\end{equation}
	Let $p', q' \in (M, \tau)$ be the following projections: 
	\begin{equation}
		\begin{aligned}
			p' & = \chi_{\{\alpha'\}}(p) \\
			q' &= \chi_{\{\beta'\}}(q) \,.		
		\end{aligned}
	\end{equation}
\end{definition}

As a consequence, 
\begin{equation}
	\begin{aligned}
		1 - p' = \chi_{\{\alpha\}}(p) \\
		1 - q' = \chi_{\{\beta\}}(q) \,.
	\end{aligned}
\end{equation}
Hence, 
\begin{equation}
	\begin{aligned}
		p & = \alpha (1 - p') + \alpha' p' = (\alpha' - \alpha) p' + \alpha \\
		q &= \beta(1 - q') + \beta' q' = (\beta' - \beta) q' + \beta \,
	\end{aligned}
\end{equation}
or equivalently 
\begin{equation}
	\label{eqn:p'_q'}
	\begin{aligned}
		p' & = \frac{p - \alpha}{\alpha' - \alpha}  \\
		q' &= \frac{q - \beta}{\beta' - \beta} \,.
	\end{aligned}
\end{equation}

Recall the Brown measure of $X$ is defined as: 
\begin{equation}
	\mu = \frac{1}{2 \pi} \nabla^2 \log \Delta(z - X)   = \frac{1}{2 \pi} \nabla^2 \frac{1}{2} \int_{0}^{\infty} \log(x) \, d \nu_z(x) \,,
\end{equation}
where $\nu_z$ is the spectral measure of $H_z(X) = (z - X)^*(z - X)$.

To compute the Brown measure, we need to complete the following steps: 

\begin{enumerate}
	\item Compute $\nu_z$.
	\item Compute $\log \Delta(z - X) = \frac{1}{2} \int_{0}^{\infty} \log(x) \, d \nu_z(x)$.
	\item Compute $\mu = \frac{1}{2 \pi} \nabla^2 \log \Delta(z - X)$.
\end{enumerate}

In Section \ref{sec:brown_measure_compute}, we will compute the Brown measure of $X$ up to some parameters that come from the von Neumann algebra generated by $2$ projections in Subsection \ref{subsec:two_projections}. In Section \ref{sec:atoms_weights}, we use freeness of $p$ and $q$ to compute these parameters and also deduce that $X$ is non-normal. In Section \ref{sec:full_brown_measure}, we combine these results to state the Brown measure of $X$ and deduce some properties. In Section \ref{sec:further_work}, we discuss some further work that has been completed on related families of operators. 

\section{Brown Measure up to Weights and Measures}
\label{sec:brown_measure_compute}
In this section, we will describe the Brown measure of $X = p + i q$ as a convex combination of 4 atoms and another probability measure, $\mu'$. Applying the model of the von Neumann algebra generated by $2$ projections with the projections $p'$ and $q'$ defined in the previous section, observe that:
\begin{equation}
	\label{eqn:e_{ij}_p_q}
	\begin{aligned}
		e_{0 0} & = \chi_{\{\alpha\}}(p) \wedge \chi_{\{\beta\}}(q) \\
		e_{0 1} &= \chi_{\{\alpha\}}(p) \wedge \chi_{\{\beta'\}}(q) \\
		e_{1 0} &= \chi_{\{\alpha'\}}(p) \wedge \chi_{\{\beta\}}(q) \\
		e_{1 1} &= \chi_{\{\alpha'\}}(p) \wedge \chi_{\{\beta'\}}(q) \\
		e &= 1 - (e_{0 0} + e_{0 1} + e_{1 0} + e_{1 1}  ) \,.
	\end{aligned}
\end{equation}
The measure $\mu'$ depends on the measure $\nu$ in the isomorphism from the von Neumann algebra generated by $2$ projections: $(e M e, \restr{\tau}{eM e}) \isomorphic \left( M_2 \left( L^{\infty} ((0, \pi / 2), \nu) \right) , \mathbb{E}_{\nu}[ \frac{1}{n} \tr ]\right) $. The weights in the convex combination are the weights $\tau(e_{i j}), \tau(e)$ from Subsection \ref{subsec:two_projections}. We will determine the Brown measure up to determining the weights $\tau(e_{i j}), \tau(e)$ and the measure $\nu$. These parameters depend on the joint law of $p$ and $q$. In particular, the results in this section are valid for arbitrary projections $p, q \in (M, \tau)$. In Section \ref{sec:atoms_weights}, we will compute these parameters when $p$ and $q$ are freely independent. 

To compute the Brown measure of $X$, recall we need to first compute $\nu_z$, the spectral measure of $H_z(X) = (z - X)^* (z - X)$.

The result of the computation is the following: 

\begin{proposition}
	\label{prop:nu_z}
	If $e = 0$,
	\begin{equation}
		\nu_z = \tau(e_{00}) \delta_{ \Abs{z - (\alpha + i \beta)}^2  } + \tau(e_{0 1}) \delta_{\Abs{z - (\alpha + i \beta')}^2} + \tau(e_{1 0}) \delta_{\Abs{z - (\alpha'+ i \beta)}^2} + \tau(e_{11}) \delta_{\Abs{z - (\alpha'+ i \beta')}^2} \,.
	\end{equation}
	If $e \neq 0$, there exists continuous functions $\sigma_{z, 1}, \sigma_{z, 2}: (0, \pi / 2) \to [0, \infty)$ such that $\sigma_{z, 1}(\theta), \sigma_{z, 2}(\theta)$ are the singular values of the element of $\left( M_2 \left( L^{\infty} ((0, \pi / 2), \nu) \right) , \mathbb{E}_{\nu}[ \frac{1}{n} \tr ]       \right)$ corresponding to $e (z - (p + i q))e \in (e M e, \restr{\tau}{eM e})$.
	
	Then,
	\begin{equation}
		\begin{aligned}
			\nu_z & = \tau(e_{00}) \delta_{ \Abs{z - (\alpha + i \beta)}^2  } + \tau(e_{0 1}) \delta_{\Abs{z - (\alpha + i \beta')}^2} + \tau(e_{1 0}) \delta_{\Abs{z - (\alpha'+ i \beta)}^2} + \\ 
			&  \qquad \qquad \tau(e_{11}) \delta_{\Abs{z - (\alpha'+ i \beta')}^2} + \tau(e) \frac{ (\sigma_{z, 1}^2)_* (\nu) + (\sigma_{z, 2}^2)_*(\nu)  }{2} \,.
		\end{aligned}
	\end{equation}
\end{proposition}
\begin{proof}
	Let $H_z = H_z(X) = (z - X)^*(z - X)$. Since the $e_{i j}, e$ are central, mutually orthogonal projections that sum to $1$, then 
	\begin{equation}
		\begin{aligned}
			(H_z)^n &= (e_{0 0} + e_{0 1} + e_{1 0} + e_{1 1} + e) (H_z)^n (e_{0 0} + e_{0 1} + e_{1 0} + e_{1 1} + e) \\
			&= e_{00} (H_z)^n e_{00} + e_{0 1} (H_z)^n e_{0 1} + e_{1 0} (H_z)^n e_{1 0} + e_{1 1} (H_z)^n e_{1 1} + e (H_z)^n e \\
			&= (e_{00} H_z e_{00} )^n+ (e_{0 1} H_z e_{0 1})^n + (e_{1 0} H_z e_{1 0} )^n + (e_{1 1} H_z e_{1 1})^n + (e H_z e)^n \,.
		\end{aligned}
	\end{equation}
	Taking traces of both sides, 
	\begin{equation}
		\begin{aligned}
			& \tau((H_z)^n) \\ 
			& = \tau ((e_{00} H_z e_{00} )^n)+  \tau ((e_{0 1} H_z e_{0 1})^n) + \\
			& \qquad \qquad \tau((e_{1 0} H_z e_{1 0} )^n) + \tau ((e_{1 1} H_z e_{1 1})^n) + \tau((e H_z e)^n) \,.
		\end{aligned}
	\end{equation}
	From the definitions of the $e_{i j}$, 
	\begin{equation}
		\begin{aligned}
			(e_{00} H_z e_{00})^n & =  (\Abs{z - (\alpha + i \beta)}^2)^n  e_{00} \\
			(e_{01} H_z e_{01})^n & =  (\Abs{z - (\alpha + i \beta')}^2)^n  e_{01} \\
			(e_{10} H_z e_{10})^n & =  (\Abs{z - (\alpha'+ i \beta)}^2)^n  e_{10} \\
			(e_{11} H_z e_{11})^n & =  (\Abs{z - (\alpha'+ i \beta')}^2)^n  e_{11} \,.
		\end{aligned}
	\end{equation}
	Combining the previous two equations,
	\begin{equation}
		\begin{aligned}
			\tau((H_z)^n) & = \tau (e_{00}) (\Abs{z - (\alpha + i \beta)}^2)^n   \\
			& \quad +  \tau (e_{01}) (\Abs{z - (\alpha + i \beta')}^2)^n   \\
			& \quad + \tau(e_{10} ) (\Abs{z - (\alpha'+ i \beta)}^2)^n \\
			& \quad + \tau (e_{11}) (\Abs{z - (\alpha'+ i \beta')}^2)^n   \\
			& \quad +  \tau(e) \restr{\tau}{e M e}((e H_z e)^n) \,.
		\end{aligned}
	\end{equation}
	The right-hand side is the $n$-th moment of the following convex combination of probability measures:
	\begin{equation}
		\label{eqn:prop:nu_z}
		\begin{aligned}
			& \tau(e_{00}) \delta_{ \Abs{z - (\alpha + i \beta)}^2  } + \tau(e_{0 1}) \delta_{\Abs{z - (\alpha + i \beta')}^2} + \\ & \qquad + \tau(e_{1 0}) \delta_{\Abs{z - (\alpha'+ i \beta)}^2} + \tau(e_{11}) \delta_{\Abs{z - (\alpha'+ i \beta')}^2} + \tau(e) \mu_{e H_z e} \,,
		\end{aligned}
	\end{equation}
	where $\mu_{e H_z e}$ is the spectral measure of $e H_z e$ in $\left( e M e, \restr{\tau}{e M e} \right) $.
	
	As moments determine compact measures on $\R$, then this is $\nu_z$.
	
	If $e = 0$, then $\nu_z$ is the desired convex combination of atoms. 
	
	If $e \neq 0$, consider $\mu_{e H_z e}$ using the isomorphism described in Section \ref{sec:two_projections}, i.e. we proceed by assuming that $(e M e, \restr{\tau}{eM e}) = \left( M_2 \left( L^{\infty} ((0, \pi / 2), \nu) \right) , \mathbb{E}_{\nu}[ \frac{1}{n} \tr ]       \right)$. 
	
	For any $m \in M$, let $\tilde{m} = e m e \in e M e$. Then, $\tilde{H_z} = (z - (\tilde{p} + i \tilde{q}))^*(z - (\tilde{p} + i \tilde{q}))$. 
	
	Now, we claim that it is possible to choose continuous functions $\sigma_{z, 1}(\theta), \sigma_{z, 2}(\theta)$ that are the singular values of $(z - (\tilde{p} + i \tilde{q})) = e (z - (p + i q)) e$ for $\theta \in (0, \pi / 2)$. First, note that from the expressions for $\tilde{p'}, \tilde{q'}$ in Subsection \ref{subsec:two_projections} that $\tilde{H_z}$ is continuous in $\theta$. Hence, the characteristic polynomial of $\tilde{H_z}$ is a monic polynomial with coefficients that are continuous in $\theta$. From the quadratic formula, the eigenvalues of $\tilde{H_z}$ have expressions in terms of the coefficients of the characteristic polynomial of $\tilde{H_z}$. As $\tilde{H_z}$ are positive operators, then the eigenvalues are real, i.e. the discriminant of the characteristic polynomial is non-negative. Hence, the two branches of the square root of the discriminant are continuous in $\theta$. Thus, we may choose continuous expressions for the two eigenvalues of $\tilde{H_z}$. Taking square roots of these two non-negative continuous functions produces $\sigma_{z, 1}(\theta), \sigma_{z, 2}(\theta)$.
	
	Thus, 
	\begin{equation}
		\begin{aligned}
			\restr{\tau}{e M e} ((\tilde{H_z})^n) & = \int_{0}^{\pi / 2} \frac{1}{2} \tr ( (\tilde{H_z})^n  )  \, d \nu(\theta) \\
			& = \int_{0}^{\pi / 2} \frac{(\sigma_{z, 1}^2(\theta))^n + (\sigma_{z, 2}^2 (\theta))^n }{2}   \, d \nu(\theta)		\,.
		\end{aligned}
	\end{equation}
	The right-hand side is the $n$-th moment of the following measure: 
	\begin{equation}
		\frac{ (\sigma_{z, 1}^2)_* (\nu) + (\sigma_{z, 2}^2)_*(\nu)  }{2} \,.
	\end{equation}
	Since moments determine compact measures on $\R$, then
	\begin{equation}
		\mu_{e H_z e} = \frac{ (\sigma_{z, 1}^2)_* (\nu) + (\sigma_{z, 2}^2)_*(\nu)  }{2} \,.
	\end{equation}
	By combining this with (\ref{eqn:prop:nu_z}), we get the desired result: 
	\begin{equation}
		\begin{aligned}
			\nu_z & = \tau(e_{00}) \delta_{ \Abs{z - (\alpha + i \beta)}^2  } + \tau(e_{0 1}) \delta_{\Abs{z - (\alpha + i \beta')}^2} + \tau(e_{1 0}) \delta_{\Abs{z - (\alpha'+ i \beta)}^2} + \\
			& \qquad \qquad \tau(e_{11}) \delta_{\Abs{z - (\alpha'+ i \beta')}^2}  + \tau(e) \frac{ (\sigma_{z, 1}^2)_* (\nu) + (\sigma_{z, 2}^2)_*(\nu)  }{2} \, .
		\end{aligned}
	\end{equation}
\end{proof}

As an intermediate step in computing the Brown measure of $X$, we compute $\log \Delta(z - X)$:
\begin{proposition}
	\label{prop:log_delta}
	
	If  $e = 0$,
	\begin{equation}
		\begin{aligned}
			\log \Delta(z - X) 
			&= \frac{1}{2} \int_{0}^{\infty} \log(x) \, d \nu_z(x) \\
			& = \tau(e_{00}) \log \Abs{z - (\alpha + i \beta)} \\
			& \quad + \tau(e_{0 1}) \log \Abs{z - (\alpha + i \beta')} \\
			& \quad + \tau(e_{1 0}) \log \Abs{z - (\alpha'+ i \beta)} \\
			& \quad + \tau(e_{11}) \log \Abs{z - (\alpha'+ i \beta')} .
		\end{aligned}
	\end{equation}
	If $e \neq 0$, there exists continuous functions $\lambda_{1}, \lambda_{2}: (0, \pi / 2) \to \C$ such that $\lambda_{1}(\theta), \lambda_{2}(\theta)$ are the eigenvalues of the element of $\left( M_2 \left( L^{\infty} ((0, \pi / 2), \nu) \right) , \mathbb{E}_{\nu}[ \frac{1}{n} \tr ]       \right)$ corresponding to $e (p + i q) e \in   (e M e, \restr{\tau}{eM e})$.
	
	Let $\mathscr{A} = \alpha' - \alpha$ and $\mathscr{B} = \beta' - \beta$ and $\sqrt{z}$ denote the principal branch of the square root defined on $\C \setminus (- \infty, 0)$. Then, 
	\begin{equation}
		\lambda_{1}(\theta) = 
		\begin{dcases}
			\frac{\alpha + \alpha'}{2} + i \frac{\beta + \beta'}{2} - \frac{1}{2} \sqrt{\mathscr{A}^2 - \mathscr{B}^2 + 2 i \mathscr{A} \mathscr{B} \cos (2 \theta)} & \text{ when } \Abs{\mathscr{A}} \geq \Abs{\mathscr{B}} \\
			\frac{\alpha + \alpha'}{2} + i \frac{\beta + \beta'}{2} - \frac{i}{2} \sqrt{\mathscr{B}^2 - \mathscr{A}^2 - 2 i \mathscr{A} \mathscr{B} \cos (2 \theta)} & \text{ when } \Abs{\mathscr{A}} < \Abs{\mathscr{B}}
		\end{dcases}
	\end{equation}
	\begin{equation}
		\lambda_{2}(\theta) = 
		\begin{dcases}
			\frac{\alpha + \alpha'}{2} + i \frac{\beta + \beta'}{2} + \frac{1}{2} \sqrt{\mathscr{A}^2 - \mathscr{B}^2 + 2 i \mathscr{A} \mathscr{B} \cos (2 \theta)} & \text{ when } \Abs{\mathscr{A}} \geq \Abs{\mathscr{B}} \\
			\frac{\alpha + \alpha'}{2} + i \frac{\beta + \beta'}{2} + \frac{i}{2} \sqrt{\mathscr{B}^2 - \mathscr{A}^2 - 2 i \mathscr{A} \mathscr{B} \cos (2 \theta)} & \text{ when } \Abs{\mathscr{A}} < \Abs{\mathscr{B}} .
		\end{dcases}
	\end{equation}	
	Then,
	\begin{equation}
		\begin{aligned}
			\log \Delta(z - X) 
			&= \frac{1}{2} \int_{0}^{\infty} \log(x) \, d \nu_z(x) \\
			& = \tau(e_{00}) \log \Abs{z - (\alpha + i \beta)} \\
			& \quad + \tau(e_{0 1}) \log \Abs{z - (\alpha + i \beta')} \\
			& \quad + \tau(e_{1 0}) \log \Abs{z - (\alpha'+ i \beta)} \\
			& \quad + \tau(e_{11}) \log \Abs{z - (\alpha'+ i \beta')} \\ 
			& \quad + \tau(e) \int_{0}^{\pi / 2} \frac{\log \Abs{z - \lambda_{1}(\theta)} + \log \Abs{z - \lambda_{2}(\theta)}}{2} \, d \nu(\theta) \,.
		\end{aligned}
	\end{equation}
\end{proposition}
\begin{proof}
	When $e = 0$, the formula for $\log \Delta(z - X)$ follows easily from Proposition \ref{prop:nu_z}.
	
	For $e \neq 0$, from Proposition \ref{prop:nu_z}, for any continuous $f: [0, \infty) \to \C$,
	\begin{equation}
		\begin{aligned}
			\int_{0}^{\infty} f(x) \, d \nu_z(x) 
			& = \tau(e_{00}) f(\Abs{z - (\alpha + i \beta)}^2) \\
			& \quad + \tau(e_{0 1}) f(\Abs{z - (\alpha + i \beta')}^2) \\
			& \quad  + \tau(e_{1 0}) f(\Abs{z - (\alpha'+ i \beta)}^2) \\
			& \quad + \tau(e_{11}) f(\Abs{z - (\alpha'+ i \beta')}^2) \\
			& \quad + \tau(e) \int_{0}^{\infty} f(x) \, d \left(  \frac{(\sigma_{z, 1}^2)_* (\nu) + (\sigma_{z, 2}^2)_*(\nu)}{2} \right) \,.
		\end{aligned}
	\end{equation}
	Rewriting the final integral using the change of variables formula,
	\begin{equation}
		\begin{aligned}
			& \int_{0}^{\infty} f(x) \, d \left(  \frac{(\sigma_{z, 1}^2)_* (\nu) + (\sigma_{z, 2}^2)_*(\nu)}{2} \right) \\
			&= \frac{1}{2} \left( \int_{0}^{\infty} f(x) \, d  (\sigma_{z, 1}^2)_* (\nu) +    \int_{0}^{\infty} f(x) (\sigma_{z, 2}^2)_*(\nu)  \right) \\
			&= \frac{1}{2} \left( \int_{0}^{\pi / 2} f(\sigma_{z, 1}^2(\theta)) \, d \nu(\theta) + \int_{0}^{\pi / 2} f(\sigma_{z, 2}^2(\theta)) \, d \nu(\theta) \right) \\
			&= \frac{1}{2} \int_{0}^{\pi / 2}  f(\sigma_{z, 1}^2(\theta)) +  f(\sigma_{z, 2}^2(\theta)) \, d \nu(\theta) \,.
		\end{aligned}
	\end{equation}
	Let $f_n(x) = \log(x + 1 / n) / 2$ for $n = 1, 2, \ldots$. Then, $f_n: [0, \infty) \to \R$ are continuous and decrease to $\log (x) / 2$. Applying two previous equations for $f = f_n$ and using the monotone convergence theorem to take the limit as $n \to \infty$,
	\begin{equation}
		\label{eqn:prop:log_delta}
		\begin{aligned}
			\log \Delta(z - X) 
			&= \frac{1}{2} \int_{0}^{\infty} \log(x) \, d \nu_z(x) \\
			& = \tau(e_{00}) \log \Abs{z - (\alpha + i \beta)} \\
			& \quad + \tau(e_{0 1}) \log \Abs{z - (\alpha + i \beta')} \\
			& \quad + \tau(e_{1 0}) \log \Abs{z - (\alpha'+ i \beta)} \\
			& \quad + \tau(e_{11}) \log \Abs{z - (\alpha'+ i \beta')} \\ 
			& \quad + \frac{\tau(e)}{2} \int_{0}^{\pi / 2} \frac{\log (\sigma_{z, 1}^2(\theta)) + \log (\sigma_{z, 2}^2(\theta))}{2} \, d \nu(\theta) \,.
		\end{aligned}
	\end{equation}
	For the rest of the proof, assume that $(e M e, \restr{\tau}{eM e}) = \left( M_2 \left( L^{\infty} ((0, \pi / 2), \nu) \right) , \mathbb{E}_{\nu}[ \frac{1}{n} \tr ]       \right)$.
	
	Recall that $\sigma_{z, 1}(\theta), \sigma_{z, 2}(\theta)$ are the singular values of $e(z - X)e$. Then, the integrand of the final integral can be simplified as: 
	\begin{equation}
		\begin{aligned}
			\frac{\log (\sigma_{z, 1}^2(\theta)) + \log (\sigma_{z, 2}^2(\theta))}{2}
			&= \frac{\log (\sigma_{z, 1}^2(\theta) \sigma_{z, 2}^2(\theta))}{2} \\ 
			&= \frac{\log ( \det ((e(z - X)e)^* e(z - X)e   )    )}{2} \\
			&= \frac{\log ( \Abs{\det ((e(z - X)e))} ^2)}{2} \\
			& = \log \Abs{\det ((e(z - X)e))} \\
			& = \log \Abs{\det (z - e X e)} .
		\end{aligned}
	\end{equation}
	Given that we can verify the formulas for $\lambda_1(\theta)$, $\lambda_2(\theta)$, then $z - \lambda_1(\theta), z - \lambda_2(\theta)$ are the eigenvalues for $z - e X e$. Thus, combining the previous two equations, 
	\begin{equation}
		\begin{aligned}
			& \int_{0}^{\pi / 2} \frac{\log (\sigma_{z, 1}^2(\theta)) + \log (\sigma_{z, 2}^2(\theta))}{2} \, d \nu(\theta) \\
			&=  \int_{0}^{\pi / 2} \log \Abs{\det (z - e X e)} \, d \nu(\theta) \\
			&= \int_{0}^{\pi / 2} \log \Abs{z - \lambda_{1}(\theta)} + \log \Abs{z - \lambda_{2}(\theta)} \, d \nu(\theta) \,.
		\end{aligned}
	\end{equation}
	Substituting this expression into (\ref{eqn:prop:log_delta}) produces the desired formula for $\log \Delta(z - X)$.
	
	We return to verifying the formulas for $\lambda_1(\theta)$, $\lambda_2(\theta)$. Straightforward computation shows that the characteristic polynomial of $e X e \in \left( M_2 \left( L^{\infty} ((0, \pi / 2), \nu) \right) , \mathbb{E}_{\nu}[ \frac{1}{n} \tr ]       \right)$ is:
	\begin{equation}
		p(\lambda) = \lambda^2 - ( (\alpha + \alpha') + i (\beta + \beta')  ) \lambda + \left( \alpha \alpha' - \beta \beta' + \frac{i}{2}  ( (\alpha + \alpha')(\beta + \beta') - \mathscr{A}\mathscr{B} \cos(2 \theta))\right).
	\end{equation}
	The eigenvalues of $e X e$ are: 
	\begin{equation}
		\frac{\alpha + \alpha'}{2} + i \frac{\beta + \beta'}{2} \pm \frac{1}{2} \sqrt{ \mathscr{A}^2 - \mathscr{B}^2 + 2 i \mathscr{A} \mathscr{B} \cos(2 \theta) } \,.
	\end{equation}
	where the square root is any branch of the square root. 
	
	When $\Abs{\mathscr{A}} \geq \Abs{\mathscr{B}}$, then $\Re(\mathscr{A}^2 - \mathscr{B}^2 + 2 i \mathscr{A} \mathscr{B} \cos(2 \theta)) \geq 0$ and the expression for the eigenvalues of $e (p + i q) e$ is continuous and well-defined if we take the square root to be the principal branch of the square root. This gives the formulas for $\lambda_1(\theta), \lambda_2(\theta)$ for $\Abs{\mathscr{A}} \geq \Abs{\mathscr{B}}$.
	
	When $\Abs{\mathscr{A}} < \Abs{\mathscr{B}}$, then $\pm i \sqrt{ \mathscr{B}^2 - \mathscr{A}^2 - 2 i \mathscr{A} \mathscr{B} \cos ( 2 \theta)   }$ are also expressions for the square roots of $\mathscr{A}^2 - \mathscr{B}^2 + 2 i \mathscr{A} \mathscr{B} \cos(2 \theta)$, where the square root is the principal branch of the square root. Since $\Abs{\mathscr{A}} < \Abs{\mathscr{B}}$, then $\Re(\mathscr{B}^2 - \mathscr{A}^2 - 2 i \mathscr{A} \mathscr{B} \cos ( 2 \theta)) > 0$ and this expression is continuous and well-defined.. This gives the formulas for $\lambda_1(\theta), \lambda_2(\theta)$ for $\Abs{\mathscr{A}} < \Abs{\mathscr{B}}$.
\end{proof}

Finally, we compute the Brown measure of $X$, $\mu = \frac{1}{2\pi} \nabla^2 \log \Delta(z - X)$ in the following Proposition: 

\begin{proposition}
	\label{prop:brown_measure_p+iq_no_nu} 
	Let $\mu$ be the Brown measure of $X$.
	If $e = 0$,
	\begin{equation}
		\begin{aligned}
			\mu &= \frac{1}{2 \pi} \nabla^2 \log \Delta(z - X) \\
			& = \tau(e_{00}) \delta_{ \alpha + i \beta } + \tau(e_{0 1}) \delta_{\alpha + i \beta'} + \tau(e_{1 0}) \delta_{\alpha'+ i \beta} + \tau(e_{11}) \delta_{\alpha'+ i \beta'} .
		\end{aligned}
	\end{equation}
	If $e \neq 0$, let $\lambda_{1}, \lambda_{2}: (0, \pi / 2) \to \C$ be as in Proposition \ref{prop:log_delta}. Then,
	\begin{equation}
		\begin{aligned}
			\mu &= \frac{1}{2 \pi} \nabla^2 \log \Delta(z - X) \\
			& = \tau(e_{00}) \delta_{ \alpha + i \beta } + \tau(e_{0 1}) \delta_{\alpha + i \beta'} + \tau(e_{1 0}) \delta_{\alpha'+ i \beta} + \tau(e_{11}) \delta_{\alpha'+ i \beta'} \\ 
			& \quad + \tau(e) \mu' \,,
		\end{aligned}
	\end{equation}
	where 
	\begin{equation}
		\mu' = \frac{ (\lambda_{1})_* (\nu) + (\lambda_{2})_*(\nu)  }{2} \,.
	\end{equation}
	Additionally, $\mu' ( \{  \alpha + i \beta, \alpha' + i \beta, \alpha + i \beta', \alpha' + i \beta'  \}  ) = 0$.
\end{proposition}
\begin{proof}
	If $e = 0$, the result follows from directly applying $\frac{1}{2 \pi} \nabla^2 \log \Abs{\cdot - \lambda} = \delta_\lambda$   to the expression for $\log \Delta(z - X)$ in Proposition \ref{prop:log_delta}.
	
	If $e \neq 0$, we take the distributional Laplacian of the result from Proposition \ref{prop:log_delta}: 
	\begin{equation}
		\begin{aligned}
			\frac{1}{2 \pi} \nabla^2 \log \Delta(z - X) 
			&= \frac{1}{2 \pi}\nabla^2 \frac{1}{2} \int_{0}^{\infty} \log(x) \, d \nu_z(x) \\
			& = \frac{1}{2 \pi}\nabla^2 \biggl(  \tau(e_{00}) \log \Abs{z - (\alpha + i \beta)} \\
			& \quad + \tau(e_{0 1}) \log \Abs{z - (\alpha + i \beta')} \\
			& \quad + \tau(e_{1 0}) \log \Abs{z - (\alpha'+ i \beta)} \\
			& \quad + \tau(e_{11}) \log \Abs{z - (\alpha'+ i \beta')} \\ 
			& \quad + \tau(e) \int_{0}^{\pi / 2} \frac{\log \Abs{z - \lambda_{1}(\theta)} + \log \Abs{z - \lambda_{2}(\theta)}}{2} \, d \nu(\theta) \biggr) .
		\end{aligned}
	\end{equation}
	As in the case when $e = 0$, applying $\frac{1}{2 \pi} \nabla^2 \log \Abs{\cdot - \lambda} = \delta_\lambda$ directly to the first $4$ atomic terms of $\log \Delta(z - X)$ produces the weighted sum of the $4$ atoms in $\mu$.
	
	To apply $\frac{1}{2 \pi} \nabla^2 \log \Abs{\cdot - \lambda} = \delta_\lambda$ for the final integral, we apply Fubini's theorem. Consider $f \in C_c^{\infty}(\C)$. Since $\log \Abs{z - w} \in L^1_{\text{loc}}(\C)$ for all $w \in \C$, then we may apply Fubini's theorem for $i = 1, 2$: 
	\begin{equation}
		\begin{aligned}
			& \brackets{f, \frac{1}{2 \pi} \nabla^2  \int_{0}^{\pi / 2} \log \Abs{z - \lambda_{i}(\theta)} \, d \nu(\theta)  } \\
			& = \brackets{\nabla^2 f, \frac{1}{2 \pi}\int_{0}^{\pi / 2} \log \Abs{z - \lambda_{i}(\theta)} \, d \nu(\theta) } \\
			&= \int_{\C}^{} \nabla^2 f(z) \left( \frac{1}{2 \pi}\int_{0}^{\pi / 2} \log \Abs{z - \lambda_{i}(\theta)} \, d \nu(\theta)\right) \, d \lambda(z) \\
			&= \int_{0}^{\pi / 2} \left( \int_{\C}^{} \frac{1}{2 \pi}  \nabla^2 f(z)  \log \Abs{z - \lambda_{i}(\theta)} \, d d \lambda(z)\right)  \, d \nu(\theta)  \\
			&= \int_{0}^{\pi / 2} f(\lambda_i(\theta))  \, d \nu(\theta) \,.
		\end{aligned} 
	\end{equation} 
	Hence, for $i = 1, 2$,
	\begin{equation}
		\frac{1}{2 \pi} \nabla^2 \int_{0}^{\pi / 2} \log \Abs{z - \lambda_{i}(\theta)} \, d \nu(\theta) = (\lambda_i)_*(\nu) \,.
	\end{equation}
	Thus, the Laplacian of the integral is: 
	\begin{equation}
		\frac{1}{2 \pi} \nabla^2 \int_{0}^{\pi / 2} \frac{\log \Abs{z - \lambda_{1}(\theta)} + \log \Abs{z - \lambda_{2}(\theta)}}{2} \, d \nu(\theta)
		= \frac{ (\lambda_{1})_* (\nu) + (\lambda_{2})_*(\nu)  }{2} 
		= \mu' .
	\end{equation}
	Combining this with the atomic terms gives the desired Brown measure for $X$.
	
	For the final note, it follows from Lemma \ref{lem:nu_0_1} and the fact that $\lambda_i(\{0, \pi / 2\}) = \{\alpha + i \beta, \alpha' + i \beta, \alpha + i \beta', \alpha' + i \beta'\}$ that
	\begin{equation}
		\mu' ( \{  \alpha + i \beta, \alpha' + i \beta, \alpha + i \beta', \alpha' + i \beta'  \}  ) 
		= \nu ( \{0, \pi / 2\}  ) = 0 \,.
	\end{equation}
\end{proof}

Even though the weights $\tau(e_{i j}), \tau(e)$ and the measure $\nu$ have not been determined yet, we can already say something about the support of the Brown measure in general. 

First, we prove a lemma about a relevant hyperbola and rectangle: 

\begin{lemma}
	\label{lem:hyperbola_rectangle}
	Let $\alpha, \alpha', \beta, \beta' \in \R$, where $\alpha \neq \alpha'$ and $\beta \neq \beta'$. Let $\mathscr{A} = \alpha' - \alpha$ and $\mathscr{B} = \beta' - \beta$.
	
	Let 
	\begin{equation}
		\begin{aligned}
			H & = \left\lbrace z = x + i y \in \C :	\left( x - \frac{\alpha + \alpha'}{2}  \right)^2 - \left(  y - \frac{\beta + \beta'}{2} \right)^2 = \frac{\mathscr{A}^2 - \mathscr{B}^2}{4}    \right\rbrace \\
			R & =  \left\lbrace z = x + i y \in \C : x \in  [\alpha \wedge \alpha', \alpha \vee \alpha' ], y \in [\beta \wedge \beta', \beta \vee \beta'] \right\rbrace \,.
		\end{aligned}
	\end{equation}		
	The equation of $H$ is equivalent to: 
	\begin{equation}
		(x - \alpha)(x - \alpha') = (y - \beta)(y - \beta') \,.
	\end{equation}
	The equation of $H$ in coordinates 
	\begin{equation}
		\begin{aligned}
			x' & = x - \frac{\alpha + \alpha'}{2} \\
			y' &= y - \frac{\beta + \beta'}{2}
		\end{aligned}
	\end{equation}
	is 
	\begin{equation}
		\label{eqn:lem:hyperbola_rectangle}
		(x')^2 - \frac{\mathscr{A}^2  }{4} = (y')^2 - \frac{\mathscr{B}^2}{4} .
	\end{equation}
	It follows that for $(x, y) \in H$, 
	\begin{equation}
		(x, y) \in R \iff (\ref{eqn:lem:hyperbola_rectangle}) \leq 0 \iff x \in  [\alpha \wedge \alpha', \alpha \vee \alpha' ] \text{ or } y \in [\beta \wedge \beta', \beta \vee \beta'] \,.
	\end{equation}
	Similarly, 
	\begin{equation}
		(x, y) \in \interior{R} \iff (\ref{eqn:lem:hyperbola_rectangle}) < 0 \iff x \in  (\alpha \wedge \alpha', \alpha \vee \alpha') \text{ or } y \in (\beta \wedge \beta', \beta \vee \beta') \,.
	\end{equation}	
	Alternatively, the equation of the hyperbola is:
	\begin{equation}
		\Re \left( \left( z - \frac{\alpha + \alpha'}{2} - i \frac{\beta + \beta'}{2}  \right)^2 \right) =   \frac{\mathscr{A}^2 - \mathscr{B}^2}{4} \,.
	\end{equation}
	If $z \in H$, then $z \in R$ if and only if
	\begin{equation}
		\Abs{\Im \left( \left( z - \frac{\alpha + \alpha'}{2} - i \frac{\beta + \beta'}{2}  \right)^2\right)}  \leq \frac{\Abs{\mathscr{A} \mathscr{B}}}{2} \,.
	\end{equation}
\end{lemma}
\begin{proof}
	The equivalent equations for the hyperbola are straightforward to check. 
	
	The equivalences for the closed conditions follow from the following equivalences and the equation of the hyperbola in $x', y'$ coordinates
	\begin{equation}
		\begin{aligned}
			(x, y) \in R & \iff (x')^2 - \frac{\mathscr{A}^2}{4} \leq 0 \text{ and } (y')^2 - \frac{\mathscr{B}^2}{4} \leq 0 \\
		\end{aligned}
	\end{equation}
	\begin{equation}
		\begin{aligned}
			(x')^2 - \frac{\mathscr{A}^2  }{4} \leq 0  & \iff \Abs{x'} \leq \frac{\Abs{\mathscr{A}}}{2} & \iff x \in [\alpha \wedge \alpha', \alpha \vee \alpha' ]  \\
			(y')^2 - \frac{\mathscr{B}^2}{4} \leq 0 & \iff  \Abs{y'} \leq \frac{\Abs{\mathscr{B}}}{2} & \iff y \in [\beta \wedge \beta', \beta \vee \beta']  \,.
		\end{aligned}
	\end{equation}
	The equivalences for the open conditions follow from similar equivalences with the closed conditions replaced by open conditions.
	
	The last equation of the hyperbola follows from direct computation. For the inequality of the rectangle, observe that
	\begin{equation}
		\Im \left( \left( z - \frac{\alpha + \alpha'}{2} - i \frac{\beta + \beta'}{2}  \right)^2  \right)  = 2 x' y' \,.
	\end{equation} 
	In light of what was previously shown, 
	\begin{equation}
		x' y' \leq \frac{\Abs{\mathscr{A} \mathscr{B}}}{4} \Longrightarrow x' \leq \frac{\Abs{\mathscr{A}}}{2} \text{ or } y' \leq \frac{\Abs{\mathscr{B}}}{2} \Longrightarrow z \in R \,.
	\end{equation}
	Conversely, 
	\begin{equation}
		z \in R \Longrightarrow x' \leq \frac{\Abs{\mathscr{A}}}{2} \text{ and } y' \leq \frac{\Abs{\mathscr{B}}}{2} \Longrightarrow x' y' \leq \frac{\Abs{\mathscr{A} \mathscr{B}}}{4} \,.
	\end{equation}
\end{proof}

Now, we will show that the support of the Brown measure is contained in $H \cap R$:

\begin{corollary}
	\label{cor:brown_measure_support}
	Let $\mathscr{A} = \alpha' - \alpha$ and $\mathscr{B} = \beta' - \beta$.
	
	The continuous functions $\lambda_1, \lambda_2: [0, \pi / 2] \to \C$ in Proposition \ref{prop:log_delta} parameterize the intersection of the hyperbola 
	\begin{equation}
		H = \left\lbrace z = x + i y \in \C :	\left( x - \frac{\alpha + \alpha'}{2}  \right)^2 - \left(  y - \frac{\beta + \beta'}{2} \right)^2 = \frac{\mathscr{A}^2 - \mathscr{B}^2}{4}    \right\rbrace 
	\end{equation}
	with the rectangle 
	\begin{equation}
		R = \left\lbrace z = x + i y \in \C : x \in  [\alpha \wedge \alpha', \alpha \vee \alpha' ], y \in [\beta \wedge \beta', \beta \vee \beta'] \right\rbrace  \,.
	\end{equation}
	When $\Abs{\mathscr{A}} \geq \Abs{\mathscr{B}}$, $\lambda_1$ parameterizes the left component of $H \cap R$ and $\lambda_2$ parameterizes the right component of $H \cap R$. When $\Abs{\mathscr{A}} < \Abs{\mathscr{B}}$, $\lambda_1$ parameterizes the bottom component of $H \cap R$ and $\lambda_2$ parameterizes the top component of $H \cap R$.
	
	The support of the Brown measure is contained in $H \cap R$.
\end{corollary}
\begin{proof}
	From Lemma \ref{lem:hyperbola_rectangle}, $z = x + i y$ is on $H \cap R$ if and only if 
	\begin{equation}
		\left( z - \frac{\alpha + \alpha'}{2} - i \frac{\beta + \beta'}{2}  \right)^2 = \frac{\mathscr{A}^2 - \mathscr{B}^2 + 2 i \mathscr{A} \mathscr{B} \cos (2 \theta)}{4} \quad \text{ for } \theta \in [0, \pi / 2] \,.
	\end{equation}
	Note that $\lambda_1(\theta), \lambda_2(\theta)$ are exactly the solutions to this equation. Hence, the $\lambda_i(\theta)$ parameterize the intersection of the hyperbola and rectangle. 
	
	For the cases of $\lambda_i$ parameterizing the left/right or top/bottom components, it is easy to see from the formulas that the $\lambda_i$ map into the left/right or top/bottom components, and since the $\lambda_i(\theta)$ parameterize all of $H \cap R$, then the $\lambda_i$ have to parameterize the entire left/right or top/bottom components. 
	
	As the $\lambda_i$ parameterize $H \cap R$, then $\mu'$ is supported on $H \cap R$. The 4 atoms in the Brown measure are on $(\partial R) \cap H$ (at the 4 corners of the rectangle). 
	
	Thus, we conclude that the Brown measure is supported on $H \cap R$.
\end{proof}

Figure \ref{fig:support_example} illustrates the conclusion of Corollary \ref{cor:brown_measure_support} where the Brown measure of $X = p + i q$ is approximated by the ESD of $X_n = P_n + i Q_n$ for deterministic $P_n, Q_n \in M_n(\C)$.

\begin{figure}[H]
	\centering
	\includegraphics[width = .5 \textwidth]{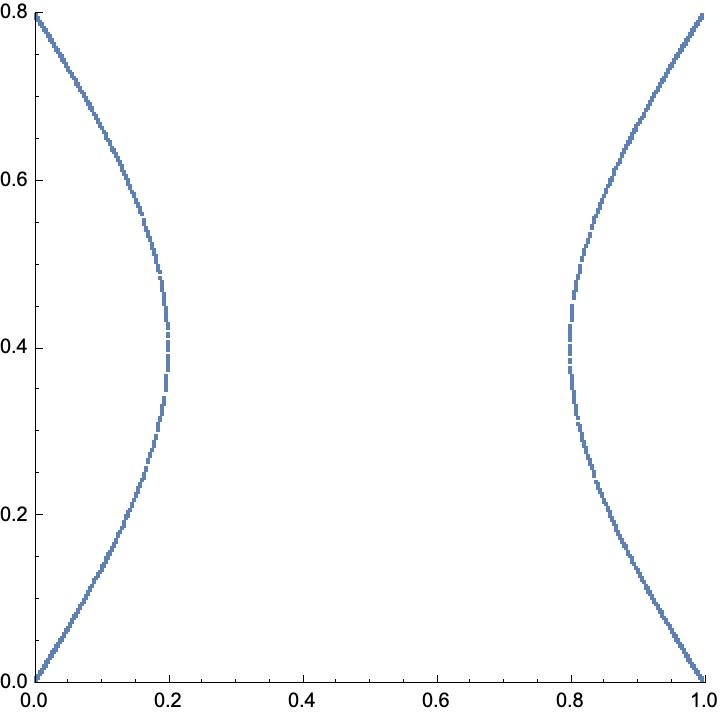}
	\caption{ESD of $X_n = P_n + i Q_n$ \\ $\mu_{P_n} = (1/2) \delta_0 + (1/2) \delta_{1}$ \\ $\mu_{Q_n} = (1 / 2) \delta_0 + (1/2) \delta_{4/5}$ \\ $n = 1000$}
	\label{fig:support_example}
\end{figure}

Motivated by the hyperbola and rectangle appearing in Corollary \ref{cor:brown_measure_support}, we introduce the definition of the hyperbola and rectangle associated with $X$: 

\begin{definition}
	\label{def:hyperbola_rectangle}
	Let $p, q \in (M, \tau)$ be Hermitian with laws: 
	\begin{equation}
		\begin{aligned}
			\mu_p & = a \delta_\alpha + (1 - a) \delta_{\alpha'} \\
			\mu_q &= b \delta_\beta + (1 - b) \delta_{\beta'} \,.
		\end{aligned}
	\end{equation}
	Let $\mathscr{A} = \alpha' - \alpha$ and $\mathscr{B} = \beta' - \beta$. 
	
	The \textbf{hyperbola associated with $X$} is 
	\begin{equation}
		H = \left\lbrace z = x + i y \in \C :	\left( x - \frac{\alpha + \alpha'}{2}  \right)^2 - \left(  y - \frac{\beta + \beta'}{2} \right)^2 = \frac{\mathscr{A}^2 - \mathscr{B}^2}{4}    \right\rbrace \,.
	\end{equation}
	The \textbf{rectangle associated with $X$} is
	\begin{equation}
		R =  \left\lbrace z = x + i y \in \C : x \in  [\alpha \wedge \alpha', \alpha \vee \alpha' ], y \in [\beta \wedge \beta', \beta \vee \beta'] \right\rbrace  \,.
	\end{equation}
\end{definition}

While the weights $\tau(e_{i j}), \tau(e)$ are as of yet undetermined in the case $p$ and $q$ are free, there are some general relationships between the weights and traces of the spectral projections of $p$ and $q$:

\begin{proposition}
	\label{prop:atoms_weights_relation}
	Let $\mu$ be the Brown measure of $X$ and let $\mu'$ be as in Proposition \ref{prop:brown_measure_p+iq_no_nu}.
	
	Then, the $\mu'$ measure of each of the two components of $H \cap R$ is equal to $1 / 2$. Additionally,
	\begin{equation}
		\begin{aligned}
			\tau \left(\chi_{ \{\alpha\}} (p) \right) & = \mu( \{\alpha + i \beta, \alpha + i \beta' \}  ) + \tau(e) / 2 \\
			\tau \left(\chi_{ \{\alpha'\}} (p) \right) & = \mu( \{\alpha' + i \beta, \alpha' + i \beta' \}  ) + \tau(e) / 2 \\
			\tau \left(\chi_{ \{\beta\}} (q) \right) & = \mu( \{\alpha + i \beta, \alpha' + i \beta \}  ) + \tau(e) / 2 \\
			\tau \left(\chi_{ \{\beta'\}} (q) \right) & = \mu( \{\alpha + i \beta', \alpha' + i \beta' \}  ) + \tau(e) / 2 \,.
		\end{aligned}
	\end{equation}
	If $\Abs{\mathscr{A}} \geq \Abs{\mathscr{B}}$, the Brown measure of the left component of $H \cap R$ is $\tau(\chi_{ \{ \alpha \wedge \alpha'  \}(p)})$ and the Brown measure of the right component of $H \cap R$ is $\tau(\chi_{  \{\alpha \vee \alpha' \}}(p))$.
	
	If $\Abs{\mathscr{A}} < \Abs{\mathscr{B}}$, the Brown measure of the bottom component of $H \cap R$ is $\tau(\chi_{ \{\beta \wedge \beta'\}}(q))$ and the Brown measure of the top component of $H \cap R$ is $\tau(\chi_{ \{\beta \vee \beta'\}  }(q) )$.
\end{proposition}
\begin{proof}
	Recall that 
	\begin{equation}
		\mu' = \frac{ (\lambda_{1})_* (\nu) + (\lambda_{2})_*(\nu)  }{2} \,,
	\end{equation}
	where $\nu$ is a probability measure on $(0, \pi / 2)$. From Corollary \ref{cor:brown_measure_support}, the $\lambda_i$ each parameterize one of the components of $H \cap R$. Hence, the $\mu'$ measure of each component of $H \cap R$ is $1/ 2$.
	
	For the equations of the traces of spectral projections of $p$ and $q$, we will just prove the first equation, the others are similar. From (\ref{eqn:e_{ij}_p_q}), 
	\begin{equation}
		\chi_{\{a\}}(p) = e_{0 0} + e_{0 1} + e \chi_{\{a\}}(p) e \,.
	\end{equation}
	From Proposition \ref{prop:brown_measure_p+iq_no_nu}, $\mu(\{\alpha + i \beta\}) = \tau(e_{0 0})$ and $\mu(\{\alpha + i \beta'\}) = \tau(e_{0 1})$. From Lemma \ref{lem:trace_1/2}, $\tau(e \chi_{\{a\}}(p) e) = \tau(e) / 2$. The desired equation follows from taking the trace of the equation and using these facts. 
	
	For the last point, we consider the Brown measure of the left component of $H \cap R$ when $\Abs{\mathscr{A}} \geq \Abs{\mathscr{B}}$ as the other cases are similar. Let $L$ be this left component. Then, 
	\begin{equation}
		\begin{aligned}
			\mu(L) & = \mu(\{ \alpha \wedge \alpha' + i \beta, \alpha \wedge \alpha' + i \beta' \}) + \tau(e) \mu'(L) \\
			& = \mu(\{ \alpha \wedge \alpha' + i \beta, \alpha \wedge \alpha' + i \beta' \}) + \tau(e) / 2 \\
			& = \tau(\chi_{ \{\alpha \wedge \alpha'\}}(p) ) \,.
		\end{aligned} 
	\end{equation}
\end{proof}

\section{Computation of Weights and Measure}
\label{sec:atoms_weights}
In this section, we will determine the weights $\tau(e_{i j}), \tau(e)$ and the measure $\nu$ in Proposition \ref{prop:brown_measure_p+iq_no_nu} under the assumption that $p$ and $q$ are freely independent. We also deduce that the operators of the form $X = p + i q$, where $p, q$ Hermitian, freely independent, and have $2$ atoms in their spectra are not normal (Corollary \ref{cor:projections_not_normal}).

We will use the functions $\psi_\mu$, $\chi_\mu$, and $S_\mu$ that were introduced in Subsection \ref{subsec:free_prob_functions}. Recall that for $x \in (M, \tau)$, we will use $\psi_x$, $\chi_x$, and $S_x$ to denote the respective functions with respect to $\mu_x$, the spectral measure of $x$.

Recall that the measure $\nu$ and the weights $\tau(e_{i j}), \tau(e)$ are computed in terms of projections $p', q'$ where
\begin{equation}
	\begin{aligned}
		p' & = \chi_{\{\alpha'\}}(p) \\
		q' &= \chi_{\{\beta'\}}(q) \,.
	\end{aligned}
\end{equation}
Since $p$ and $q$ are freely independent, then $p'$ and $q'$ are also freely independent. This along with the traces $\tau(p') = 1 - a$ and $\tau(q') = 1 - b$ determine the joint law of $p'$, $q'$. 

\textbf{For notational convenience, in this section we will consider two general projections $p$ and $q$ that are freely independent and $\tau(p) = a$, $\tau(q) = b$ for some $a, b \in (0, 1)$. This is a natural continuation of Subsection \ref{subsec:two_projections}. It is easy to translate the results in this section to the $p', q'$ defined in Section \ref{sec:notation}. }

Recall from Subsection \ref{subsec:two_projections} that the weights $\tau(e_{i j}), \tau(e)$ are: 
\begin{equation}
	\begin{aligned}
		\tau(e_{0 0}) & = \tau((1 - p) \wedge (1 - q)) \\
		\tau(e_{0 1}) &= \tau((1 - p) \wedge q) \\
		\tau(e_{1 0}) &= \tau(p \wedge (1 - q)) \\
		\tau(e_{1 1}) &= \tau(p \wedge q) \\
		\tau(e) &= 1 - (\tau(e_{0 0}) + \tau(e_{0 1}) + \tau(e_{1 0}) + \tau(e_{1 1})  ) \,.
	\end{aligned}
\end{equation}
Recall that when $e \neq 0$, $((0, \pi / 2), \nu)$ is the pushforward measure of $((0, 1), \nu^*)$ under the inverse of of $\cos^2(\theta)$ and $\nu^*$ is the spectral measure of $ e x e$, where $x = pqp + (1 - p)(1 - q)(1 - p)$. 

Since $\tau((pqp)^n) \to \tau(p \wedge q)$ as $n \to \infty$, then understanding the laws of $p q p$ and $(1 - p)(1 - q)(1 - p)$ are relevant for computing both $\nu$ and the weights $\tau(e_{i j}), \tau(e)$. The first step to this is computing the relevant free probability functions: 

\begin{proposition}
	\label{prop:chi_psi_s}
	Let $p, q \in (M, \tau)$ be two freely independent projections with $\tau(p) = a$, $\tau(q) = b$, $a, b \in (0, 1)$. Then, 
	\begin{equation}
		\begin{aligned}
			\psi_{p}(z) = \frac{a z}{1 - z} & \qquad  \psi_q(z) = \frac{b z}{1 - z} \\
			\chi_p(w) = \frac{w}{w + a} & \qquad \chi_q(w) = \frac{w}{w + b}
		\end{aligned}
	\end{equation}
	\begin{equation}
		\chi_{p q p}(w) = \frac{w (1 + w)}{(w + a)(w + b)} \,.
	\end{equation}
	Let 
	\begin{equation}
		f(z) = 1 + (4 ab - 2(a + b)) z + (a - b)^2 z^2 \,.
	\end{equation}
	Then, $\psi_{p q p}, \psi_{(1 - p)(1 - q)(1 - p)}$ are analytic on $\C \setminus [1, \infty)$ and
	\begin{equation}
		\begin{aligned}
			\psi_{p q p}(z) & = \frac{1 - (a + b) z - \sqrt{f(z)}}{2 (z - 1)} \\
			\psi_{(1 - p) (1 - q) (1 - p)}(z) & = \frac{1 - (2 - a - b) z - \sqrt{f(z)}}{2 (z - 1)} \, , 
		\end{aligned}
	\end{equation}
	where $\sqrt{f(z)}$ is an analytic branch of the square root of $f(z)$ on $\C \setminus [1, \infty)$ where $\sqrt{1} = + 1$. In particular, the root(s) of $f(z)$ are in $[1, \infty)$ and are distinct when $f$ is quadratic. 
\end{proposition}
\begin{proof}
	The formula for $\psi_p$ follows from 
	\begin{equation}
		\mu_p = (1 - a) \delta_0 + a \delta_1
	\end{equation}
	and similarly for $\psi_q$. Since $\psi_p$, $\psi_q$ are linear fractional transformations, the formulas for their inverses ($\chi_p$, $\chi_q$, respectively) are easily computed.
	
	Since $p$ and $q$ are freely independent, then the formula for $\chi_{pqp}$ follows from the multiplicativity of the $S$-transform and its relationship with $\chi$.
	
	Since $\norm{p qp } \leq 1$ then $\mu_{p q p}$ is supported on $[0, 1]$.  So, the formula 
	\begin{equation}
		\psi_{p q p }(z) = \int_{0}^{\infty} \frac{t z}{1 - t z} \, d \mu_{p q p}(t) = \int_{0}^{1} \frac{t z}{1 - t z} \, d \mu_{p q p}(t)
	\end{equation}
	defines an analytic function for $z \in \C \setminus [1, \infty)$.
	
	To compute the formula for $\psi_{p q p}$, recall that $\psi_{p q p} = \chi_{pqp}^{-1}$ in a neighborhood of $0$. Then, 
	\begin{equation}
		z = \chi_{p q p}(w) = \frac{w (1 + w)}{(w + a)(w + b)}
	\end{equation}
	if and only if
	\begin{equation}
		z(w + a)(w + b) = w(1 + w) \,.
	\end{equation}
	This is clearly true for $w \neq -a, -b$ and the latter equation is not satisfied for $w = -a, -b \in (-1, 0)$ at any $z$.
	
	This last equation is true if and only if
	\begin{equation}
		(z - 1) w^2 + (z(a + b) - 1)w + z a b = 0 \,.
	\end{equation}
	Fixing a $z$ and solving for $w$, then from the quadratic formula and simplifying, we obtain the desired formula for $\psi_{p q p}$. The sign of the square root follows from the general fact that $\psi_{\mu}(0) = 0$.
	
	The formula for $\psi_{(1 - p)(1 - q)(1 - p)}$ follows from the formula for $\psi_{p q p}$ and noting that $1 - p$ and $1 - q$ are freely independent projections with traces $\tau(p) = 1 - a$, $\tau(q) = 1 - b$. Observe that $f(z)$ is invariant under changing the pair $(a, b)$ to $(1 - a, 1 - b)$.
	
	Finally, note that the $\sqrt{f(z)}$ in both $\psi_{pqp}$ and $\psi_{(1 - p)(1 - q)(1 - p)}$ are identical since they are defined on the domain $\C \setminus [1, \infty)$ and agree at $z = 0$.
	
	The fact that $f(z)$ has root(s) in $[1, \infty)$ follows from the fact that $\sqrt{w}$ is not analytic in any neighborhood of $0$, so the root(s) of $f$ cannot be on $\C \setminus [1, \infty)$.
	
	For distinctness of the roots when $f$ is quadratic, the discriminant is $16 a b (1 - a) ( 1 - b) > 0$ for $a, b \in (0, 1)$. 
\end{proof}

Now, we proceed to determine $\tau(e_{i j}), \tau(e)$. First, we need the following Lemma \cite[Proposition 8]{SpeicherBook}:

\begin{lemma}
	\label{lem:G_mu_point}
	Let $\mu$ be a finite measure on the real line and $s \in \R$. For a sequence $z_n \to s$ non-tangentially to $\R$, $(z_n - s)G_\mu(z_n) \to \mu (\{s\})$.
\end{lemma}
\begin{proof}
	Let $z_n = x_n + i y_n$. The condition that $z_n \to s$ non-tangentially to $\R$ is equivalent to $\Abs{(x_n - s) / y_n} \leq M$ for some $M$. In particular, $y_n \neq 0$.
	
	Computation shows that
	\begin{equation}
		(z_n - s) G_\mu(z_n) = \int_{- \infty}^{\infty} \frac{z_n - s}{z_n - t} \, d \mu(t) \,.
	\end{equation}
	Let $f_n: \R \to \C$ where 
	\begin{equation}
		f_n(t) = \frac{z_n - s}{z_n - t} \,.
	\end{equation}
	Note that $f_n(s) = 1$ for all $n$ and $\lim\limits_{n \to \infty} f_n(t) = 0$ for all $t \neq s$. It suffices to show that the $f_n(t)$ are uniformly bounded, as then the result follows from the bounded convergence theorem.
	
	First, rewrite $f_n(t)$: 
	\begin{equation}
		f_n(t) = \frac{z_n - s}{z_n - t} = \frac{(x_n - s) + i y_n}{(x_n - t) + i y_n} = \frac{(x_n - s) / y_n + i}{(x_n - t) / y_n + i} \,.
	\end{equation}
	Since $\Abs{(x_n - s) / y_n + i} \leq \Abs{(x_n - s) / y_n} + 1 = 1 + M$ and $\Abs{(x_n - t) / y_n + i} \geq 1$, 
	\begin{equation}
		\Abs{f_n(t)} = \frac{\Abs{(x_n - s) / y_n + i}}{\Abs{(x_n - t) / y_n + i}} \leq \frac{1 + M}{1} = 1 + M \,.
	\end{equation}
	as desired.
\end{proof}

Next, we determine $\tau(e_{i j}), \tau(e)$ using the free independence of $p$ and $q$:

\begin{proposition}
	\label{prop:atoms_weights}
	Let $p, q \in (M, \tau)$ be two freely independent projections with $\tau(p) = a$, $\tau(q) = b$, $a, b \in (0, 1)$. Then, 
	\begin{equation}
		\begin{aligned}
			\tau(e_{0 0}) & = \tau((1 - p) \wedge (1 - q)) = \max(0, (1 - a)  + (1 - b) - 1) \\
			\tau(e_{0 1}) &= \tau((1 - p) \wedge q) = \max(0, (1 - a) + b - 1) \\
			\tau(e_{1 0}) &= \tau(p \wedge (1 - q)) = \max(0, a + (1 - b) - 1  )\\
			\tau(e_{1 1}) &= \tau(p \wedge q) = \max(0, a + b - 1) \\
			\tau(e) &= 1 - (\tau(e_{0 0}) + \tau(e_{0 1}) + \tau(e_{1 0}) + \tau(e_{1 1})  ) \,.
		\end{aligned}
	\end{equation}	
\end{proposition}
\begin{proof}
	By replacing $p$ with $1 - p$ and/or $q$ with $1 - q$, it suffices to just prove 
	\begin{equation}
		\tau(p \wedge q) = \max(0, a + b - 1) \,.
	\end{equation}
	Recall that $\tau((pqp)^n) \to \tau(p \wedge q)$ as $n \to \infty$. Since $x^n \to \chi_{\{1\}}(x)$ on $[0, 1]$ and $\sigma(pqp) \subset [0, 1]$, then $\tau((pqp)^n) \to \mu_{pqp}(\{1\})$. Hence,
	\begin{equation}
		\tau(p \wedge q) = \mu_{pqp}(\{1\}) \,.
	\end{equation} 
	We proceed to use Proposition \ref{prop:chi_psi_s} and Lemma \ref{lem:G_mu_point} to determine $\mu_{pqp}(\{1\})$.
	
	In general, for $z \in \C \setminus \sigma(pqp)$,
	\begin{equation}
		G_{p q p}(z) = \frac{1}{z} \left( \psi_{p q p } \left( \frac{1}{z} \right) + 1 \right) \,.
	\end{equation}
	From Proposition \ref{prop:chi_psi_s}, the right-hand side of this equation is: 
	\begin{equation}
		\frac{1}{z} \left( \psi_{p q p } \left( \frac{1}{z} \right) + 1 \right) = \frac{z + (a + b - 2) + z \sqrt{f(1 / z)}}{2 z} \,.
	\end{equation}
	and is defined on $\C \setminus [0, 1]$. Since $\sigma(pqp) \subset [0, 1]$, then $G_{p q p}$ is also defined on $\C \setminus [0, 1]$, so then the following equality holds for $z \in \C \setminus [0, 1]$:
	\begin{equation}
		G_{pqp}(z) = \frac{z + (a + b - 2) + z \sqrt{f(1 / z)}}{2 z (z - 1) } \,.
	\end{equation}
	Thus, we may use this formula and Lemma \ref{lem:G_mu_point} to obtain: 
	\begin{equation}
		\begin{aligned}
			\mu_{pqp}(\{1\})
			&=  \lim\limits_{z \to 1} (z - 1) G_{p q p}(z) \\
			&= \lim\limits_{z \to 1} \frac{z + (a + b - 2) + z \sqrt{f(1 / z)}}{2 z} \\ 
			&= \frac{a + b - 1 + \Abs{a + b - 1}}{2} \\
			&= \max(0, a + b - 1) \,.
		\end{aligned}
	\end{equation}
\end{proof}

Proposition \ref{prop:atoms_weights} can be summarized by: ``free projections intersect as little as possible.'' For $ \tau(p \wedge q) = \max(0, a + b - 1)$, the term $\max(0, a + b - 1)$ is just the minimum trace of the intersection between two projections $p$ and $q$ where of $\tau(p) = a$ and $\tau(q) = b$: 

Recall from the parallelogram law that for projections $p, q \in (M, \tau)$,
\begin{equation}
	\tau(p \wedge q) = \tau(p) + \tau(q) - \tau(p \vee q)   \geq \tau(p) - \tau(q) - 1 = a + b - 1 \,.
\end{equation}
As $\tau(p \wedge q) \geq 0$, then for any projections $p, q \in (M, \tau)$, $\tau(p \wedge q) \geq \max(0, a + b - 1)$.

As a corollary, we will no longer need to consider the possibility that $e = 0$: 

\begin{corollary}
	\label{cor:e_not_0}
	Let $p, q \in (M, \tau)$ be two freely independent projections. Then, $e = 0$ if and only if one of $p, 1 - p, q, 1 - q$ is $0$.
\end{corollary}
\begin{proof}
	If one of $p, 1 - p, q, 1 - q$ is $0$, then it is clear that $e = 0$.
	
	Suppose that none of $p, 1 - p, q, 1 - q$ is $0$. Then, we may apply Proposition \ref{prop:atoms_weights}. As $\tau$ is faithful, it suffices to conclude that $\tau(e) = 0$. 
	
	Consider the expressions of $a$ and $b$ in Proposition \ref{prop:atoms_weights}. Observe that the expressions inside $\tau(e_{00})$ and $\tau(e_{1 1})$ are negatives of each other and hence
	\begin{equation}
		\tau(e_{00}) + \tau(e_{1 1}) = \max(0, - (a + b - 1)) + \max(0, a + b - 1) = \Abs{a + b - 1} .
	\end{equation}
	Similarly, 
	\begin{equation}
		\tau(e_{0 1}) + \tau(e_{1 0}) = \max(0, b - a) + \max(0, a - b ) = \Abs{a - b} .
	\end{equation}
	Thus,
	\begin{equation}
		\tau(e) = 0 \iff \Abs{a + b - 1} + \Abs{a - b} = 1 \,.
	\end{equation}
	By considering the four cases of $\Abs{a + b - 1} = \pm (a + b - 1)$ and $\Abs{a - b} = \pm (a - b)$, we see that $\Abs{a + b - 1} + \Abs{a - b} = 1$ if and only if one of $p, 1 - p, q, 1 - q$ is $0$. Hence, $\tau(e) \neq 0$.
\end{proof}

As a corollary, we deduce the fact that the operators of the form $X = p + i q$ are not normal: 

\begin{corollary}
	\label{cor:projections_not_normal}
	Let $p, q \in (M, \tau)$ be Hermitian, freely independent, and have $2$ atoms in their spectra. Then, $X = p + i q$ is not normal.
\end{corollary}
\begin{proof}
	First, we consider the case when $p$ and $q$ are projections with $\tau(p) = a$, $\tau(q) = b$, $a, b \in (0, 1)$.
	
	It suffices to show that $p$ and $q$ do not commute. For this, recall that $p$ and $q$ commute if and only if $pq = p \wedge q$. Applying this to the other three pairs of commuting projections, $(1 - p, q)$, $(p, 1 - q)$, and $(1 - p, 1 - q)$, then $(1 - p)q = (1 - p) \wedge q$, $p(1 - q) = p \wedge (1 - q)$, and $(1 - p)(1 - q) = (1 - p) \wedge (1 - q)$. 
	
	Thus if $p$ and $q$ commute, the following equalities hold: 
	\begin{equation}
		\begin{aligned}
			1 & = (p + (1 - p))(q + (1 - q)) \\
			& = pq + (1 - p)q + p(1 - q) + (1 - p)(1 - q) \\
			& = p \wedge q + (1 - p) \wedge q + p \wedge (1 - q) + (1 - p) \wedge (1 - q) \,.
		\end{aligned}
	\end{equation}
	
	In the notation of the projections $e_{i j}, e$, this means that $e = 0$. From Corollary \ref{cor:e_not_0}, this cannot happen if $a, b \in (0, 1)$. Thus, $X = p + i q$ is not normal. 
	
	The general case follows by considering the projections $p'$ and $q'$ defined in Section \ref{sec:notation}. Then, $X = p + i q$ is normal if and only if $X' = p' + i q'$ is normal, since $p$ and $q$ commute if and only if $p'$ and $q'$ commute. From the previous case, $X'$ is not normal. Hence, $X$ is not normal. 
\end{proof}

We proceed to compute $\nu$ with the knowledge that $e \neq 0$. This requires computing $\nu^*$, the spectral measure of $ e x e $ in $(e M e, \restr{\tau}{e M e})$, where $x = pqp + (1 - p)(1 - q)(1 - p)$, and then pushing forward the measure under the inverse of $\cos^2$ onto $(0, \pi / 2)$. The result of the computation is the following: 

\begin{proposition}
	\label{prop:nu_measure}
	Let $p, q \in (M, \tau)$ be two freely independent projections with $\tau(p) = a, \tau(q) = b$, $a, b \in (0, 1)$. Let $f: \R \to \R$ be given by:
	\begin{equation}
		f(x) = 1 + (4a b - 2(a + b))x + (a - b)^2 x^2 \,.
	\end{equation}
	The measure $\nu$ on $(0, \pi / 2)$ is a probability measure with density with respect to Lebesgue measure $\theta$: 
	\begin{equation}
		\frac{d \nu}{d \theta} = \frac{2}{\pi} \frac{1}{\tau(e)} \Im \left( \sqrt{f(\sec^2 \theta)} \right) \cot(\theta) \, ,
	\end{equation}
	where the square root of a negative number is on the positive imaginary axis.  
\end{proposition}
\begin{proof}
	To compute $\nu^*$, first note that 
	\begin{equation}
		\begin{aligned}
			pqp & = p \wedge q + e (pqp) e \\
			(1 - p)(1 - q)(1 - p) &= (1 - p) \wedge (1 - q) + e ((1 - p)(1 - q)(1 - p)) e \,.
		\end{aligned}
	\end{equation}
	Taking $n$-th powers of each side, 
	\begin{equation}
		\begin{aligned}
			(pqp)^n & = (p \wedge q)^n + (e (pqp) e)^n \\
			((1 - p)(1 - q)(1 - p))^n &= ((1 - p) \wedge (1 - q))^n + (e (1 - p)(1 - q)(1 - p) e)^n .
		\end{aligned}
	\end{equation}
	Computing $\psi_{e (pqp) e}, \psi_{e (1 - p)(1 - q)(1 - p) e}$ in $(e M e, \restr{\tau}{eMe})$ and using the power series expansions of general $\psi_\mu$, then in a neighborhood of $0$,
	\begin{equation}
		\begin{aligned}
			\psi_{p q p}(z) & = \psi_{p \wedge q}(z) + \tau(e) \psi_{e (p q p ) e}(z)  \\
			\psi_{(1 - p)(1 - q)(1 - p)}(z) &= \psi_{(1 - p) \wedge (1 - q)}(z) + \tau(e) \psi_{e (1 - p) (1 - q) (1 - p)  e}(z) \,.
		\end{aligned}
	\end{equation}
	Since the spectra of  $pqp, e (pqp) e, (1 - p)(1 - q)(1 - p), e (1 - p)(1 - q)(1 - p)  e, p \wedge q, (1 - p) \wedge (1 - q)$ are all contained in $[0, 1]$, then this equality holds on $\C \setminus [1, \infty)$.
	
	From Corollary \ref{cor:e_not_0}, $e \neq 0$, so 
	\begin{equation}
		\begin{aligned}
			\psi_{e (pqp) e}(z) & = \frac{\psi_{pqp}(z) - \psi_{p \wedge q}(z)}{\tau(e)} \\
			&= \frac{\psi_{pqp}(z)}{\tau(e)} - \frac{\tau(p \wedge q)}{\tau(e)} \frac{z}{1 - z} \\
			\psi_{e (1 - p) (1 - q) (1 - p)  e}(z) &= \frac{\psi_{(1 - p)(1 - q)(1 - p)}(z) - \psi_{(1 - p) \wedge (1 - q)}(z) }{\tau(e)} \\
			&= \frac{\psi_{(1 - p)(1 - q)(1 - p)}(z)  }{\tau(e)} - \frac{\tau((1 - p) \wedge (1 - q))}{\tau(e)} \frac{z}{1 - z} \,.
		\end{aligned}
	\end{equation}
	Since $e$ is central and $x = pqp + (1 - p)(1 - q)(1 - p)$, then for $n \geq 1$, 
	\begin{equation}
		(exe)^n = (e(pqp)e)^n +  (e (1 - p)(1 - q) (1 - p) e  )^n .
	\end{equation}
	From Proposition \ref{prop:chi_psi_s}, in $(e M e, \restr{\tau}{e M e})$,
	\begin{equation}
		\begin{aligned}
			\psi_{exe}(z) & = \psi_{e (pqp) e}(z) + \psi_{e (1 - p)(1 - q)(1 - p) e}(z) \\
			&= \frac{1}{\tau(e)} \left( \psi_{pqp}(z) + \psi_{(1 - p)(1 - q)(1 - p)}(z) \right) \\
			& \qquad  - \frac{ \tau(p \wedge q)}{\tau(e)} \frac{z }{1 - z} - \frac{\tau((1 - p) \wedge (1 - q))}{\tau(e)} \frac{z}{1 - z} \\ 
			&= \frac{1}{\tau(e)} \left(- 1 - \frac{\sqrt{f(z)}}{z - 1}\right)  - \frac{\tau(p \wedge q) + \tau((1 - p) \wedge (1 - q))}{\tau(e)} \frac{z}{1 - z} \,,
		\end{aligned}
	\end{equation}
	where 
	\begin{equation}
		f(z) = 1  + (4 ab - 2 (a + b) ) z + (a - b)^2 z^2 .
	\end{equation}
	From Proposition \ref{prop:chi_psi_s}, $\sqrt{f(z)}$ is analytic on $\C \setminus [1, \infty)$, so this formula for $\psi_{exe}$ is valid on $\C \setminus [1, \infty)$.
	
	Hence, the following formula for $G_{\nu^*} = G_{e x e} $ is valid on $\C \setminus [0, 1]$:
	\begin{equation}
		\begin{aligned}
			G_{\nu^*}(z) & =  G_{exe}(z)   \\
			&= \frac{1}{z} \left( \psi_{exe}\left(\frac{1}{z}\right)  + 1\right) \\
			&= \frac{1}{\tau(e)} \frac{\sqrt{f(1/ z)}}{z - 1} - \frac{1}{\tau(e) z} -  \frac{\tau(p \wedge q) + \tau((1 - p) \wedge (1 - q))}{\tau(e)}  \frac{1}{z(z - 1)} \,.
		\end{aligned}
	\end{equation}
	Recall that $\nu^*$ is supported on $[0, 1]$. We observe that the measure $\nu^*$ has no atoms: For $t \in (0, 1)$, this follows from Lemma \ref{lem:G_mu_point} and computing that $\lim\limits_{z \to t} (z - t) G_{\nu^*}(z) = 0$. Recall from Subsection \ref{subsec:two_projections} that $\nu^* (\{0, 1\}) = 0$ is true in general. Hence, $\nu^*$ has no atoms. 
	
	Thus, we may recover the measure completely with the formula (Proposition \ref{prop:stieltjes_intervals}): 
	\begin{equation}
		\begin{aligned}
			\lim\limits_{y \to 0^+} \int_{a}^{b} -\frac{1}{\pi} \Im \, G_{\nu^*} (x + i y)  \, d x 
			& = \nu^*((a, b)) + \frac{1}{2}  \left( \nu^*(\{a\}) + \nu^*(\{b\})\right)  \\
			& =  \nu^*([a, b]) \,.
		\end{aligned}
	\end{equation}
	where $a, b \in (0, 1)$.
	
	On compact subsets of $(0, 1)$, $G_{\nu^*}(x + i y)$ is uniformly bounded as $y \to 0^+$. Given that the following pointwise limit exists for $x \in (0, 1)$:
	\begin{equation}
		\lim\limits_{y \to 0^+} -\frac{1}{\pi} \Im \, G_{\nu^*} (x + i y) \, ,
	\end{equation}
	then this limit will be the density of $\nu^*$ with respect to the Lebesgue measure (see (\ref{eqn:stieltjes_vague_limit})). 
	
	Using the formula for $G_{\nu^*}(z)$, we notice that for any $x \in (0, 1)$, only one term is non-zero in this limit: 
	\begin{equation}
		\begin{aligned}
			\lim\limits_{y \to 0^+} -\frac{1}{\pi} \Im \, G_{\nu^*} (x + i y) 
			&= \lim\limits_{y \to 0^+}  \frac{1}{\pi} \Im \left( \frac{1}{\tau(e)} \frac{\sqrt{f(1 / (x + i y))}}{1 - (x + i y) }   \right)
			\\ &= \frac{1}{\pi \tau(e)} \frac{\lim\limits_{y \to 0^+}  \Im \sqrt{f(1 / (x + i y))} }{1 - x} \,.
		\end{aligned}
	\end{equation}
	We proceed by showing $\lim\limits_{y \to 0^+}  \Im \sqrt{f(1 / (x + i y))}$ exists. It suffices to show the following limit exists:
	\begin{equation}
		\lim\limits_{\substack{z \to t  \\ \Im(z) < 0} } g(z) \, ,
	\end{equation}
	where $t \in (0, 1)$ and $g$ is defined an analytic on the lower half-plane and $g(z)^2 = f(z)$. First, note that for any sequence $z_n \to t$, $g(z_n)$ is bounded, because $g(z_n)^2 = f(z_n)$. Considering a convergent subsequence where $g(z_{n_k})$ converges, then if $g(z_{n_k}) \to s$, then $s^2 = \lim\limits_{n_k \to \infty} g(z_{n_k})^2 = \lim\limits_{n_k \to \infty} f(z_{n_k}) = f(t)$. Hence, $g(z_{n_k})$ converges to a square root of $f(t)$. When $f(t) = 0$ this is enough to prove the desired limit. When $f(t) \neq 0$, we consider if there are two sequences in the lower half-plane, $z_n, z_n' \to t$, where  $g(z_n), g(z_n')$ converge to the different square roots of $f(t)$. Taking derivatives of $g(z)^2 = f(z)$, then $2 g(z) g'(z) = f'(z)$. As $\Abs{f'(z)}$ is bounded from above and $\Abs{g(z)}$ bounded from below near $t$ where $f(t) \neq 0$, then $\Abs{g'(z)}$ is bounded from above near $t$. Then, $g(z)$ is Lipschitz near $t$. Thus, $g(z_n), g(z_n')$ converging to different square roots is a contradiction.
	
	Since the measure is positive, then $\lim\limits_{y \to 0^+}  \Im \sqrt{f(1 / (x + i y))}$ has to be non-negative, at least Lebesgue almost everywhere $t \in (0, 1)$. It is possible to extend this to all $t \in (0, 1)$ by noting this set of $t$ is dense and using $g$ being Lipschitz near $t$ where $f(t) \neq 0$.
	
	Thus, for $x \in (0, 1)$,
	\begin{equation}
		\frac{d \nu^*}{d \lambda} = \frac{1}{\pi \tau(e)} \frac{\Im \sqrt{f(1 / x)}}{1 - x} \, ,
	\end{equation}
	where the square root of a negative number is on the positive imaginary axis. 
	
	By applying the change of variables formula, the measure $\nu$ on $(0, \pi / 2)$ is given by: 
	\begin{equation}
		\frac{d \nu}{d \theta} = \frac{2}{\pi} \frac{1}{\tau(e)} \Im \left( \sqrt{f(\sec^2 \theta)} \right) \cot \theta \,.
	\end{equation}
	Note that this is a probability measure, being the spectral measure of a non-zero element of $e M e$.
\end{proof}

\section{The Brown measure of $X$}
\label{sec:full_brown_measure}
By combining Propositions \ref{prop:brown_measure_p+iq_no_nu}, \ref{prop:atoms_weights}, and \ref{prop:nu_measure}, we can state the Brown measure of $X$:

\begin{theorem}
	\label{thm:brown_measure_p+iq}
	Let $p, q \in (M, \tau)$ be Hermitian, freely independent and
	\begin{equation}
		\begin{aligned}
			\mu_p & = a \delta_\alpha + (1 - a) \delta_{\alpha'} \\
			\mu_q &= b \delta_\beta + (1 - b) \delta_{\beta'} \,,
		\end{aligned}
	\end{equation}
	where $a, b \in (0, 1)$, $\alpha \neq \alpha'$, $\beta \neq \beta'$, and $\alpha, \alpha', \beta, \beta' \in \R$. 	
	
	Let 
	\begin{equation}
		\begin{aligned}
			\epsilon_{0 0} & = \max(0, a  + b - 1) \\
			\epsilon_{0 1} & = \max(0, a + (1 - b) - 1) \\
			\epsilon_{1 0} & = \max(0, (1 - a) + b - 1  )\\
			\epsilon_{1 1} & = \max(0, (1 - a) + (1 - b) - 1) \\
			\epsilon &= 1 - (\epsilon_{0 0} + \epsilon_{0 1} + \epsilon_{1 0} + \epsilon_{1 1}  ) \,.
		\end{aligned}
	\end{equation}	
	Then, $\epsilon > 0$.
	
	Let $\mathscr{A} = \alpha' - \alpha$ and $\mathscr{B} = \beta' - \beta$ and $\sqrt{z}$ denote the principal branch of the square root defined on $\C \setminus (- \infty, 0)$. Then, 
	\begin{equation}
		\lambda_{1}(\theta) = 
		\begin{dcases}
			\frac{\alpha + \alpha'}{2} + i \frac{\beta + \beta'}{2} - \frac{1}{2} \sqrt{\mathscr{A}^2 - \mathscr{B}^2 + 2 i \mathscr{A} \mathscr{B} \cos (2 \theta)} & \text{ when } \Abs{\mathscr{A}} \geq \Abs{\mathscr{B}} \\
			\frac{\alpha + \alpha'}{2} + i \frac{\beta + \beta'}{2} - \frac{i}{2} \sqrt{\mathscr{B}^2 - \mathscr{A}^2 - 2 i \mathscr{A} \mathscr{B} \cos (2 \theta)} & \text{ when } \Abs{\mathscr{A}} < \Abs{\mathscr{B}}
		\end{dcases}
	\end{equation}
	\begin{equation}
		\lambda_{2}(\theta) = 
		\begin{dcases}
			\frac{\alpha + \alpha'}{2} + i \frac{\beta + \beta'}{2} + \frac{1}{2} \sqrt{\mathscr{A}^2 - \mathscr{B}^2 + 2 i \mathscr{A} \mathscr{B} \cos (2 \theta)} & \text{ when } \Abs{\mathscr{A}} \geq \Abs{\mathscr{B}} \\
			\frac{\alpha + \alpha'}{2} + i \frac{\beta + \beta'}{2} + \frac{i}{2} \sqrt{\mathscr{B}^2 - \mathscr{A}^2 - 2 i \mathscr{A} \mathscr{B} \cos (2 \theta)} & \text{ when } \Abs{\mathscr{A}} < \Abs{\mathscr{B}} \,.
		\end{dcases}
	\end{equation}	
	Let $f: \R \to \R$ be given by:
	\begin{equation}
		f(x) = 1 + (4 a b - 2(a + b))x + (a - b)^2 x^2 \,.
	\end{equation}
	Let $\nu$ be a probability measure on $(0, \pi / 2)$ with density with respect to Lebesgue measure $\theta$: 
	\begin{equation}
		\frac{d \nu}{d \theta} = \frac{2}{\pi} \frac{1}{\epsilon} \Im \left( \sqrt{f(\sec^2 \theta)} \right) \cot(\theta) \, ,
	\end{equation}
	where the square root of a negative number is on the positive imaginary axis. 
	
	Let $\mu'$ be a complex probability measure given by:
	\begin{equation}
		\mu' = \frac{(\lambda_1)_*(\nu) + (\lambda_2)_*(\nu) }{2} \,.
	\end{equation}
	Then, the Brown measure of $X = p + i q$ is:
	\begin{equation}
		\mu = \epsilon_{0 0} \delta_{ \alpha + i \beta } + \epsilon_{0 1} \delta_{\alpha + i \beta'} + \epsilon_{1 0} \delta_{\alpha'+ i \beta} + \epsilon_{1 1} \delta_{\alpha'+ i \beta'} + \epsilon \mu' \,.
	\end{equation}
\end{theorem}
\begin{proof}
	The Theorem follows from combining several previous results: 
	
	The general form of $\mu$ is given in Proposition \ref{prop:brown_measure_p+iq_no_nu}. The $\tau(e_{ i j}), e$ are relabeled as $\epsilon_{ i j}, \epsilon$ in light of Proposition \ref{prop:atoms_weights} and interchanging general $p$, $q$ for $p' = \chi_{\{\alpha'\}}(p), q' = \chi_{\{\beta'\}}(q)$. Additionally, the fact that $\epsilon \neq 0$ comes from Corollary \ref{cor:e_not_0}. Finally, the measure $\nu$ comes from Proposition \ref{prop:nu_measure} (after switching $p, q$ for $p', q'$).
\end{proof}

We observe that the measure $\nu$ is only dependent on the weights of the measures of $p$ and $q$ (i.e. $a$ and $b$), and the ``shape'' of the measure (positions of the atoms and $\lambda_i$) is only dependent on the positions of the atoms $p$ and $q$ (i.e. $\alpha, \alpha', \beta, \beta'$)

We will use the definition of $\mu'$ from Theorem \ref{thm:brown_measure_p+iq} for what follows: 

\begin{definition}
	Let $X = p + i q$, where $p, q \in (M, \tau)$ are Hermitian, freely independent, and have $2$ atoms. Define $\mu'$ to be the measure as in $\ref{thm:brown_measure_p+iq}$.
\end{definition}

Now, we make some observations about the properties of the Brown measure of $X$ in the following Corollaries. For the figures, we approximate the Brown measure of $X = p + i q$ with the empirical spectral distribution of a specific $X_n = P_n + i Q_n$, where $P_n, Q_n \in M_n(\C)$ are independently Haar-rotated Hermitian matrices with the same spectral distributions as $p$ and $q$, respectively. The justification for this approximation will be discussed in a subsequent paper, where we will prove the convergence of the empirical spectral distributions of the $X_n$ to the Brown measure of $X$.

First, we consider the atoms of the Brown measure: 

\begin{corollary}
	\label{cor:atoms}
	Let $\mu$ be the Brown measure of $p + iq$ where $p$ and $q$ have $2$ atoms. Then,
	
	\begin{enumerate}
		\item $\mu'$ does not have atoms.
		\item $\mu$ is never atomic.
		\item $\mu$ can have atoms only at the points $\alpha + i \beta, \alpha' + i \beta, \alpha + i \beta', \alpha' + i \beta' $.
		\item $\mu$ has no atoms at $\alpha + i \beta$ and $\alpha' + i \beta'$ if and only if $a + b = 1$. If $a + b \neq 1$, then $\mu$ has exactly $1$ atom at either $\alpha + i \beta$ or $\alpha' + i \beta'$, with size $\Abs{a + b - 1}$.
		\item $\mu$ has no atoms at $\alpha + i \beta'$ and $\alpha' + i \beta$ if and only if $a = b$. If $a \neq b$, then $\mu$ has exactly $1$ atom at either $\alpha + i \beta'$ or $\alpha' + i \beta$, with size $\Abs{a - b}$.
		\item $\mu$ has 0, 1, or 2 atoms. $\mu$ has 0 atoms if and only if $a = b = 1/2$. $\mu$ has 1 atom if and only if one of $a + b = 1$ or $a = b$. $\mu$ has 2 atoms if and only if $a + b \neq 1$ and $a \neq b$.
		\item Changing $a \mapsto 1 - a$ and/or $b \mapsto 1 - b$ permutes the $\epsilon_{i j}$.
	\end{enumerate}
\end{corollary}
\begin{proof}
	\begin{enumerate}
		\item $\nu$ is absolutely continuous and $\lambda_i$ are injective, so $(\lambda_i)_*(\nu)$ does not have atoms. Hence, $\mu'$ does not have atoms.
		\item Since $\epsilon \neq 0$, then $\mu' \neq 0$, so $\mu$ is never atomic
		\item Since $\mu'$ does not have atoms, the only atoms of $\mu$ can occur at the points $\alpha + i \beta, \alpha' + i \beta, \alpha + i \beta', \alpha' + i \beta' $.
		\item Note that 
		\begin{equation}
			\epsilon_{0 0} + \epsilon_{1 1} = \max(0, - (a + b - 1)) + \max(0, a + b - 1) = \Abs{a + b - 1}  .
		\end{equation}
		Hence, when $a + b = 1$, $\epsilon_{0 0} + \epsilon_{1 1} = 0$ so $\mu$ has no atoms at $\alpha + i \beta$ and $\alpha' + i \beta'$. If $a + b \neq 1$, then only one of $a + b - 1$ and $- (a + b - 1)$ is positive and equal to $\Abs{a + b - 1}$ and hence one of $\epsilon_{0 0}, \epsilon_{ 1 1}$ is equal to $\Abs{a + b - 1}$.
		\item Follows similarly to 4, with the equation 
		\begin{equation}
			\epsilon_{0 1} + \epsilon_{ 1 0} = \max(0, b - a) + \max(0, a - b ) = \Abs{a - b} .
		\end{equation}
		\item Directly follows from 4. and 5.
		\item Follows directly from the formulas for the $\epsilon_{i j}$.
	\end{enumerate}
\end{proof}

Now, we consider the symmetries of $\mu'$: 

\begin{corollary}
	\label{cor:symmetry}
	Let $\mu$ be the Brown measure of $p + iq$ where $p$ and $q$ have $2$ atoms
	
	\begin{enumerate}
		\item Swapping the roles of $p$ and $q$ fixes $\nu$.
		\item Changing one of $a \mapsto 1 - a$ or $b \mapsto 1 - b$ changes $\nu$ by changing $\theta \to \pi / 2 - \theta$. These correspond changing $p \mapsto \alpha + \alpha' - p$ and $q \mapsto \beta + \beta' - q$, respectively.
		\item Changing both $a \mapsto 1 - a$ and $b \mapsto 1 - b$ fixes $\nu$. This corresponds to changing both $p \mapsto \alpha + \alpha' - p$ and $q \mapsto \beta + \beta' - q$. This amounts to changing $p + i q \mapsto (  \alpha + \alpha' ) + i (\beta + \beta') - (p + i q)$.
	\end{enumerate}
\end{corollary}
\begin{proof}
	\begin{enumerate}
		\item Swapping the roles of $p$ and $q$ amounts to swapping $a$ and $b$ in the formula for $f(x)$. But, the formula for $f(x)$ is symmetric with respect to $a, b$, so $\nu$ is fixed. 
		\item Since $f$ is symmetric with respect to changing $a$ and $b$, it suffices to check just $a \mapsto 1 - a$. Denote $f$ by $f_{a, b}$ to refer to the coefficients. For this, we just need to check the identity: 
		\begin{equation}
			\sqrt{f_{1 - a, b}(\sec^2 \theta)} \cot \theta = \sqrt{f_{a, b} (\csc^2  \theta)  } \tan \theta 
		\end{equation}
		for $\theta \in (0, \pi / 2)$. As $\cot \theta, \tan \theta > 0$ for $\theta \in (0, \pi / 2)$, it is equivalent to check 
		\begin{equation}
			f_{1 - a, b} (\sec^2 \theta) \cot^2 \theta = f_{a, b} (\csc^2 \theta) \tan^2 \theta \,.
		\end{equation}
		Checking this identity is a straightforward calculation.
		
		\item Follows from applying 2. twice. 
	\end{enumerate}
\end{proof}

Figure \ref{fig:cor:symmetry} illustrates the behavior in Corollary \ref{cor:symmetry} where the Brown measure of $X = p + i q$ is approximated by the ESD of $X_n = P_n + i Q_n$ for deterministic $P_n, Q_n \in M_n(\C)$. Note the symmetry between the two ESDs.

\begin{figure}[H]
	\centering
	\begin{subfigure}{0.45 \textwidth}
		\centering
		\includegraphics[width = .9 \textwidth]{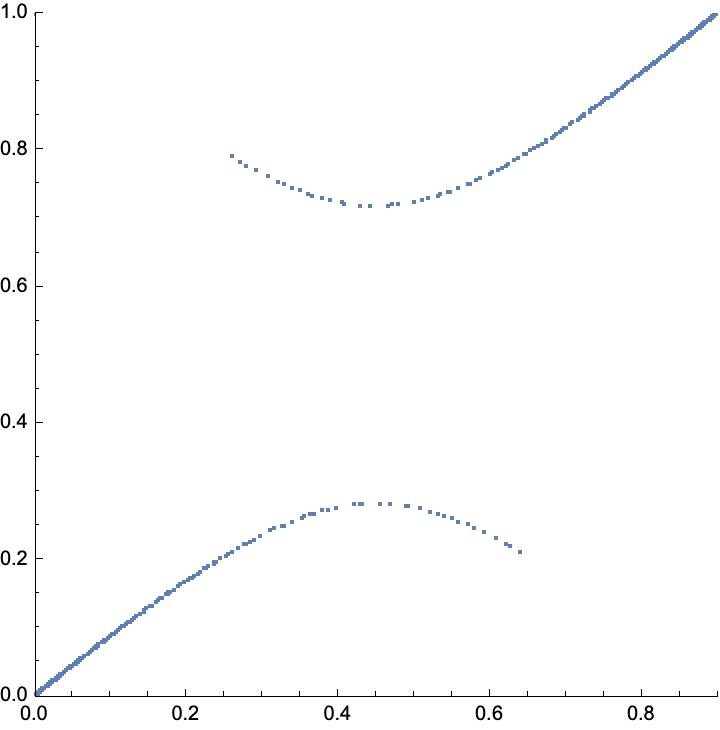}
		\caption{ESD of $X_n = P_n + i Q_n$ \\ $\mu_{P_n} = (4/5) \delta_0 + (1/5) \delta_{9/10}$ \\ $\mu_{Q_n} = (1 / 5) \delta_0 + (4/5) \delta_1$ \\ $n = 1000$}
	\end{subfigure}
	\hfill
	\begin{subfigure}{0.45 \textwidth}
		\centering
		\includegraphics[width = .9 \textwidth]{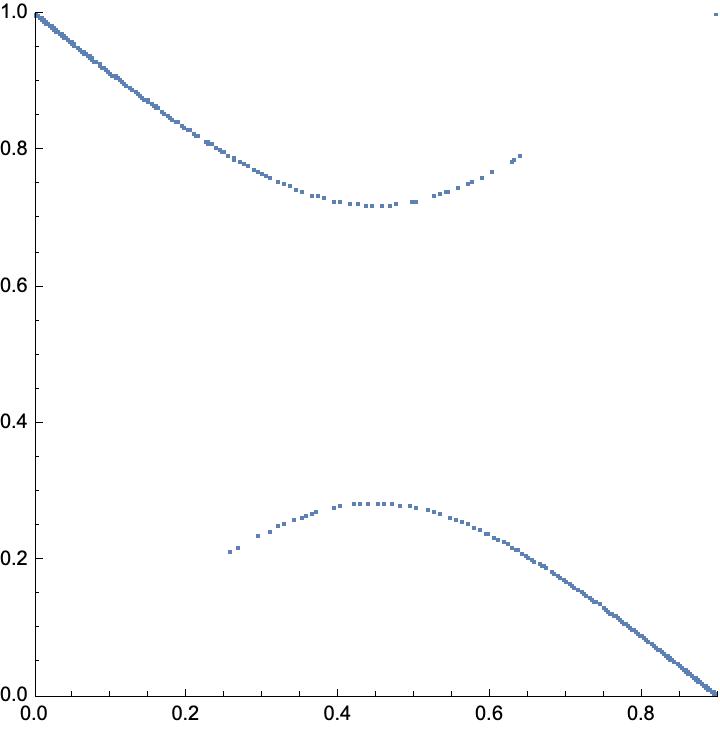}
		\caption{ESD of $X_n = P_n + i Q_n$ \\ $\mu_{P_n} = (1/5) \delta_0 + (4/5) \delta_{9/10}$ \\ $\mu_{Q_n} = (1 / 5) \delta_0 + (4/5) \delta_1$ \\ $n = 1000$}
	\end{subfigure}
	\caption{ESDs for Corollary \ref{cor:symmetry}}
	\label{fig:cor:symmetry}
\end{figure} 

Finally, we consider when the density of $\mu'$ extends all the way to the 4 corners of the intersection of the hyperbola with the boundary of the rectangle: 

\begin{corollary}	
	\label{cor:support_boundary}
	Let $\mu$ be the Brown measure of $p + i q$ where $p$ and $q$ have $2$ atoms
	\begin{enumerate}
		\item $\mu'$ has density extending to $\alpha' + i \beta$ and $\alpha + i \beta'$ if and only if $a = b$.
		\item $\mu'$ has density extending to $\alpha + i \beta$ and $\alpha' + i \beta'$ if and only if $a = 1 - b$.
		\item $\mu'$ has density extending to all $4$ corners of the intersection of the hyperbola with the boundary of the rectangle if and only if $a = b = 1/2$. Hence, the support of $\mu$ is equal to $H \cap R$ if and only if $a = b = 1/2$.
	\end{enumerate}
\end{corollary}
\begin{proof}
	Let $\mathscr{A} = \alpha' - \alpha$ and $\mathscr{B} = \beta' - \beta$.
	\begin{enumerate}
		\item For $z = \alpha' + i \beta, \alpha + i \beta'$, 
		\begin{equation}
			\left( z - \frac{\alpha + \alpha'}{2} - i \frac{\beta + \beta'}{2}  \right)^2  =   \frac{\mathscr{A}^2 - \mathscr{B}^2 - 2 i \mathscr{A} \mathscr{B}}{4} \,.
		\end{equation}
		Hence, $z = \lambda_i(\pi / 2)$.
		
		Thus, it is equivalent to determine for which $a, b$ the measure $\nu$ has density approaching $\pi / 2$. Note that $\nu$ has density approaching $\pi / 2$ if and only if $f(\sec^2 \theta) < 0$ for $\theta \to \pi / 2^-$. Note that $\lim\limits_{\theta \to \pi / 2^-} \sec^2 \theta = \infty$. Recall that $f(z)$ is a polynomial of degree at most $2$. If $f$ is quadratic, then $\lim\limits_{z \to \infty}f(z) = \infty$, so then $\nu$ does not have density approaching $\pi / 2$. Hence, we must have $a = b$. Conversely, if $a = b$, then $f(z) = 1 + 4(a^2 - a) z = 1 + 4 (a - 1)a z$, and the linear term has negative coefficient for $a \in (0, 1)$. Thus, $f(z) < 0$ as $z \to \infty$ and $\nu$ has density approaching $\pi / 2$.
		
		\item For $z = \alpha + i \beta, \alpha' + i \beta'$, 
		\begin{equation}
			\left( z - \frac{\alpha + \alpha'}{2} - i \frac{\beta + \beta'}{2}  \right)^2  =   \frac{\mathscr{A}^2 - \mathscr{B}^2 + 2 i \mathscr{A} \mathscr{B}}{4} \,.
		\end{equation}
		Hence, $z = \lambda_i(0)$.		
		
		Thus, it is equivalent to determine for which $a, b$ the measure $\nu$ has density approaching $0$. Note that $\nu$ has density approaching $0$ if and only if $f(\sec^2 \theta) < 0$ for $\theta \to 0^+$. Since $\lim\limits_{\theta \to 0^+} \sec^2 \theta = 1$, then we need that $f(x) < 0$ for $x \to 1^+$. Since $f(1) = (a + b - 1)^2$, then $f(1) = 0$, i.e. $a = 1 - b$. Conversely, if $a = 1 - b$, then $f(1) = 0$ and since $f$ is either quadratic with positive leading coefficient or $f$ is linear with negative slope, then $f$ must be negative to the right of $1$. Note that in the case where $f$ is quadratic, $f$ cannot have a double root at $1$ (from Proposition \ref{prop:chi_psi_s}).
		
		\item Follows from 1. and 2.
	\end{enumerate}	
\end{proof}

Figure \ref{fig:cor:support_boundary} illustrates the behavior in Corollary \ref{cor:support_boundary} where the Brown measure of $X = p + i q$ is approximated by the ESD of $X_n = P_n + i Q_n$, for deterministic $P_n, Q_n \in M_n(\C)$. The left ESD does not have density approaching the corners of $R$, but the right ESD does. 

\begin{figure}[H]
	\centering
	\begin{subfigure}{0.45 \textwidth}
		\centering
		\includegraphics[width = .9 \textwidth]{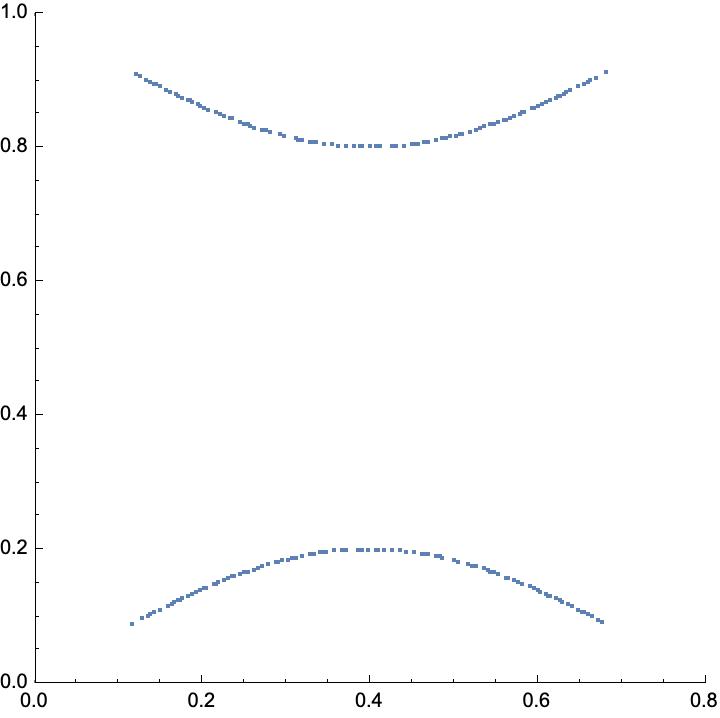}
		\caption{ESD of $X_n = P_n + i Q_n$ \\ $\mu_{P_n} = (9/10) \delta_0 + (1/10) \delta_{4/5}$ \\ $\mu_{Q_n} = (1 / 2) \delta_0 + (1/2) \delta_1$ \\ $n = 1000$}
	\end{subfigure}
	\hfill
	\begin{subfigure}{0.45 \textwidth}
		\centering
		\includegraphics[width = .9 \textwidth]{Figure_4.jpeg}
		\caption{ESD of $X_n = P_n + i Q_n$ \\ $\mu_{P_n} = (4/5) \delta_0 + (1/5) \delta_{9/10}$ \\ $\mu_{Q_n} = (1 / 5) \delta_0 + (4/5) \delta_1$ \\ $n = 1000$}
	\end{subfigure}
	\caption{ESDs for Corollary \ref{cor:support_boundary}}
	\label{fig:cor:support_boundary}
\end{figure}

Finally, we conclude that the Brown measure of $X = p + i q$ uniquely determines the laws of $p$ and $q$. Here we allow $p$ and $q$ to be possibly constant. 

\begin{corollary}
	\label{cor:brown_measure_injective}
	Let $p, q \in (M, \tau)$ be Hermitian and freely independent and
	\begin{equation}
		\begin{aligned}
			\mu_p & = a \delta_\alpha + (1 - a) \delta_{\alpha'} \\
			\mu_q &= b \delta_\beta + (1 - b) \delta_{\beta'} \,.
		\end{aligned}
	\end{equation}
	where $a, b \in [0, 1]$ and $\alpha, \alpha', \beta, \beta' \in \R$. Let $\mu$ be the Brown measure of $X = p + i q$. Then, the assignment $(\mu_p, \mu_q) \mapsto \mu$ is 1 to 1. 
\end{corollary}
\begin{proof}
	From Corollary \ref{cor:atoms}, $\mu$ is atomic if and only if one of $p$ or $q$ is constant. In this case, it is easy to determine the weights and atoms of both $\mu_p$ and $\mu_q$ from $\mu$. 
	
	Thus, we may consider $\mu$ which is the Brown measure of $X = p + i q$ where $p$ and $q$ are not constant.
	
	First, we show that $\mu$ determines the positions of the atoms of $p$ and $q$. Since the support of $\mu'$ on $H \cap R$ contains at least $5$ points, then the equation of $H$ is uniquely determined. If $a = b$ or $a + b = 1$, then from Corollary \ref{cor:support_boundary}, we can determine the positions of the atoms of $p$ and $q$. Thus, assume that $a \neq b$ and $a + b \neq 1$. From Corollary \ref{cor:atoms}, $\mu$ has $2$ atoms at the points $\{\alpha + i \beta, \alpha' + i \beta, \alpha + i \beta', \alpha' + i \beta' \}$. From these two points, at least $3$ of $\alpha, \alpha', \beta, \beta'$ are determined. To determine the last one, we can use the equation of the hyperbola and look at either the coefficient of $x$ or $y$. Thus, $\mu$ determines the positions of the atoms of $p$ and $q$. 
	
	We can determine the weights of the atoms of $p$ and $q$ directly from Proposition \ref{prop:atoms_weights_relation}.
\end{proof}

\newpage

\printbibliography

\end{document}